\documentclass[11pt,letterpaper]{article}
\usepackage[margin=1in]{geometry}
\usepackage{tikz-cd} 
\usepackage{amsmath}
\usepackage{amsthm}
\usepackage{amsfonts}
\usepackage{amssymb}
\usepackage{enumerate}
\usepackage{graphicx}
\usepackage{mathtools}
\usepackage{ytableau}

\newtheorem{thm}{Theorem}[section]
\newtheorem{defn}[thm]{Definition}
\newtheorem{lem}[thm]{Lemma}
\newtheorem{prop}[thm]{Proposition}
\newtheorem{cor}[thm]{Corollary}
\newtheorem{exmp}[thm]{Example}

\newtheorem*{convention}{Convention}

\definecolor{light-gray}{gray}{0.9}

\newcommand{\Sym}{\mathfrak{S}}

\newcommand{\len}[1]{\ell(#1)}
\renewcommand{\l}{\lambda}
\renewcommand{\a}{\alpha}
\renewcommand{\b}{\beta}

\newcommand{\w}{w}
\newcommand{\map}[3]{{#1}\,:\,{#2}\longrightarrow{#3}}
\newcommand{\Ilk}[1][k]{\mathsf{L}_{#1}}
\newcommand{\Irk}[1][k]{\mathsf{R}_{#1}}
\newcommand{\Fset}{\mathcal{F}}
\newcommand{\pt}{\partial}
\newcommand{\aascent}{ascent}

\newcommand{\bascent}{$\b$-ascent}

\DeclareMathOperator{\des}{des}
\DeclareMathOperator{\Des}{Des}
\DeclareMathOperator{\maj}{maj}
\DeclareMathOperator{\aasc}{asc}
\DeclareMathOperator{\aAsc}{Asc}
\DeclareMathOperator{\amaj}{comaj}
\DeclareMathOperator{\bmaj}{comaj_{\b}}
\DeclareMathOperator{\basc}{asc_{\b}}
\DeclareMathOperator{\bAsc}{Asc_{\b}}

\DeclareMathOperator{\inv}{inv}

\newcommand{\up}[1]{\Psi^{#1}}
\newcommand{\down}[1]{\Psi_{#1}}
\newcommand{\Aset}{\mathcal{A}}
\newcommand{\Bset}{\mathcal{B}}
\newcommand{\qbinom}[2]{\begin{bmatrix} #1 \\ #2 \end{bmatrix}}
\newcommand{\tqbinom}[2]{\left[\begin{smallmatrix} #1 \\ #2 \end{smallmatrix}\right]}
\newcommand{\pti}[2]{\pt^{#2}{#1}}
\newcommand{\dti}[2]{\dt^{#2}{#1}}
\newcommand{\km}[1]{\mu_{k}(#1)}

\newcommand{\astar}{\a_k}
\newcommand{\Astar}{\Aset_k}
\newcommand{\bstar}{\b_{k+1}}
\newcommand{\Bstar}{\Bset_{k+1}}
\newcommand{\emptyf}{(0)}
\newcommand{\Pset}{\mathcal{P}}
\newcommand{\Dset}{\mathcal{D}}
\newcommand{\qf}[3]{(#1 ; #2)_{#3}}

\newcommand{\pcnt}[1]{\vec{c}_{#1}}
\renewcommand{\vec}[1]{\mathbf{#1}}
\newcommand{\Wset}{\mathcal{R}}
\newcommand{\Wstar}{\overline{\mathcal{R}}}
\newcommand{\wstar}{\overline{\w}}

\newcommand{\aburge}[1]{\Omega_{\a,#1}}
\newcommand{\akburge}[1][k]{\Omega_{\a,#1}}
\newcommand{\akcode}[1]{$[\a,#1]$-code}
\newcommand{\bburge}[1]{\Omega_{\b,#1}}
\newcommand{\bkburge}[1][k]{\Omega_{\b,#1}}
\newcommand{\bkcode}[1]{$[\b,#1]$-code}

\newcommand{\words}[1]{\mathcal{W}_{#1}}
\newcommand{\foata}{\phi}

\newcommand{\leni}{\ell_i}

\usetikzlibrary{tikzmark}
\newcounter{pairctr}
\newcommand{\Lpair}[1]{\stepcounter{pairctr}\tikzset{tikzmark prefix=\thepairctr}\tikzmark{start}#1\tikzmark{stop}\tikz[remember picture, overlay]{
\draw[->]([shift={(.5ex,2.5ex)}]pic cs:start) to[bend left=20] ([shift={(-.5ex,2.5ex)}]pic cs:stop);}}
\newcommand{\Rpair}[1]{\stepcounter{pairctr}\tikzset{tikzmark prefix=\thepairctr}\tikzmark{start}#1\tikzmark{stop}\tikz[remember picture, overlay]{
\draw[<-]([shift={(.5ex,2.5ex)}]pic cs:start) to[bend left=20] ([shift={(-.5ex,2.5ex)}]pic cs:stop);}}
\newcommand{\BLpair}[1]{\stepcounter{pairctr}\tikzset{tikzmark prefix=\thepairctr}\tikzmark{start}#1\tikzmark{stop}\tikz[remember picture, overlay]{
\draw[->]([shift={(.5ex,2.5ex)}]pic cs:start) to[bend left=20] ([shift={(1.5ex,2.5ex)}]pic cs:stop);}}
\newcommand{\BRpair}[1]{\stepcounter{pairctr}\tikzset{tikzmark prefix=\thepairctr}\tikzmark{start}#1\tikzmark{stop}\tikz[remember picture, overlay]{
\draw[<-]([shift={(-1.5ex,2.5ex)}]pic cs:start) to[bend left=20] ([shift={(-.5ex,2.5ex)}]pic cs:stop);}}

\title{A Generalized Burge Correspondence and  $k$-measure of  Partitions }

\author{John Irving\footnote{
Department of Mathematics and Computing Science, Saint Mary's University,  Halifax, Canada, email: john.irving@smu.ca}}

\date{\today}

\begin{document}
\maketitle

\begin{abstract}
Let $\Pset$ be the set of integer partitions and $\Dset$ the subset of those with distinct parts.
We extend a correspondence of Burge~\cite{Bur-1} between partitions  and binary words to give encodings of both $\Pset$ and $\Dset$ as words over a $k$-ary alphabet, for any fixed $k\geq 2$. These  are used to prove refinements of two partition identities involving \emph{$k$-measure} that were recently derived algebraically by Andrews, Chern and Li~\cite{ACZ}. The relationship between our encoding of $\Dset$ and minimum gap-size partition identities (e.g. Schur's Theorem) is also briefly discussed.    
\end{abstract}

\newcommand{\emptyp}{\varepsilon}

\section{Introduction}

Let $\Pset$ denote the set of all \emph{partitions}, i.e. 
weakly decreasing sequences $\l=(\l_1,\l_2,\ldots)$ of nonnegative integers with finite sum $|\l|:=\l_1+\l_2+\cdots$. If $|\l|=n$ then we say $\l$ is a \emph{partition of $n$}. The nonzero $\l_i$ are called the \emph{parts} of $\l$. We write $\len{\l}$ for the number of parts, also known as the \emph{length} of $\l$.  Let $\Dset$ denote set of partitions with distinct parts.

\begin{figure}[t]
\centering
\includegraphics[height=2cm]{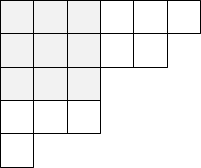}
\caption{The Young diagram of $\l=(6,5,3,3,1)$ with its Durfee square (of size 3) highlighted.}
\label{fig:young}
\end{figure}

The primary concern of this article is the following statistic on $\Pset$,  introduced recently by Andrews, Bhattacharjee and Dastida~\cite{ABD}.

\begin{defn}
\label{defn:kmeasure}
Let $k \geq 1$.
The \emph{$k$-measure} of a partition $\l=(\l_1,\l_2,\ldots)$, denoted $\km{\l}$, is the length of the longest sequence $i_1 < i_2 < \cdots < i_m$ such that $\l_{i_j} - \l_{i_{j+1}} \geq k$ for all $j$.  
\end{defn}

Equivalently, $\km{\l}$ is the maximal length of a \emph{$k$-distinct} subpartition of $\l$, i.e. one whose parts differ pairwise by at least $k$.  For example, $\l=(12,10,8,8,7,5,3,2,2,1)$ has $\mu_1(\l)=8$, $\mu_2(\l)=6$ and $\mu_3(\l)=4$, as evidenced by the following maximum-length subpartitions:
$$
\begin{array}{c|c|l}
k & \km{\l} & \text{subpartition} \\
\hline
1 & 8 & (12,10,8,7,5,3,2,1) \\
2 & 6 & (12,10,8,5,3,1) \\
3 & 4 & (12,8,5,2)
\end{array}
$$
Note that there may be several $k$-distinct subpartitions of length $\mu_k(\l)$, such as $(12,8,5,2)$ and $(12,8,5,1)$ in the case $k=3$ above.

While Definition~\ref{defn:kmeasure} appears in~\cite{ABD}, the sole focus of that paper is the case $k=2$. In particular, a surprising connection is revealed between 2-measure and another familiar statistic on partitions.   Recall that each $\l \in \Pset$ can be represented by its \emph{Young diagram}, which is as a left- and top-justified arrangement of $|\l|$ boxes in which  $\l_i$ boxes comprise the $i$-th row.  
The \emph{Durfee square} of $\l$ is the largest square that fits into the upper left corner of its Young diagram. (See Figure~\ref{fig:young}.)  The main result of~\cite{ABD} is as follows.

\begin{thm}[\cite{ABD}, Theorem 2]
\label{thm:abd}
The number of partitions of $n$ with 2-measure $m$ equals the number
of partitions of n with Durfee square of side $m$.
\end{thm}

The proof of Theorem~\ref{thm:abd} offered in~\cite{ABD}  is quite   technical. Indeed, the authors note that their approach belies the simplicity of the result and  leave the search for a bijective proof as an intriguing open problem.  However, it turns out that such a proof had  appeared much earlier in seminal work of Burge~\cite{Bur-1}, who described elegant correspondences between partitions and binary words and applied them to great effect to explain various partition identities.  More will be said on this below. 

Andrews, Chern and Li~\cite{ACZ} revisited $k$-measure for general $k$ and developed  generating series describing the distribution of $\km{\l}$ over both $\Pset$ and the subset $\Dset$ consisting of partitions with distinct parts.  Their results are best stated in terms of the \emph{$q$-Pochhammer symbol}, defined by 
$$
	\qf{a}{q}{n} := \prod_{i = 0}^n (1-aq^i)
$$ 
for $n \geq 1$ or $n=\infty$.  Note that we take a purely formal viewpoint and are not concerned with convergence; i.e.,  $\qf{a}{q}{\infty}$ and other infinite sums and products are to be regarded as residing int he appropriate ring of formal power series.  With this in mind, the main results of~\cite{ACZ} are as follows.

\begin{thm}[\cite{ACZ}, Theorem 1.5]
\label{thm:acz1}
Let $k \geq 2$. Then  
\begin{align}
\label{eq:acz1}
	\sum_{\l \in \Pset} y^{\len{\l}} z^{\km{\l}} q^{|\l|}
	&= \qf{z}{q^{k-1}}{\infty}
	\sum_{n \geq 0} \frac{z^n}{\qf{q^{k-1}}{q^{k-1}}{n}
	\qf{yq}{q}{(k-1)n}} 
\end{align}
\end{thm}

\begin{thm}[\cite{ACZ}, Theorem 1.10]
\label{thm:acz2}
Let $k \geq 2$. Then
\begin{align}
\label{eq:acz2}
	\sum_{\l \in \Dset} y^{\len{\l}} z^{\km{\l}} q^{|\l|}
	= \qf{z}{q^{k}}{\infty}
	\sum_{n \geq 0} \frac{\qf{-yq}{q}{kn}z^n}{\qf{q^{k}}{q^{k}}{n}}.
\end{align}
\end{thm}

Both~\eqref{eq:acz1} and~\eqref{eq:acz2} are proved in~\cite{ACZ} by  developing and solving $q$-difference relations satisfied by the respective partition generating series.  The proofs are succinct and straightforward, involving elementary arguments to set up the difference relations, standard $q$-series manipulations  to solve them, and an application of ``Heine's method'' to bring the solutions into the forms quoted here.  Nonetheless, given the simplicity of the results, the authors ask whether a more combinatorial  approach could be found.

The primary aim of this article
 is to extend one of Burge's aforementioned correspondences and apply it to illuminate some of the combinatorics underlying these identities.  In particular, for any fixed $k\geq 2$ we develop a bijection between partitions and  $k$-ary words whereby the size and $k$-measure of a partition are related to the  descent number (des) and major index (maj) of the associated word. (Our bijection agrees with Burge's in case $k=2$.)  A modest refinement of Theorem~\ref{thm:acz1} is then seen to follow  via MacMahon's classical enumeration of the pair (des, maj) over rearrangement  classes of words.  We  also  modify our bijection so that it applies to partitions with distinct parts,  and derive a similar refinement of Theorem~\ref{thm:acz2}  by appealing to an extension of MacMahon's result due to Clarke and Foata~\cite{ClFo-2,ClFo-3}.

To state our results more precisely we will require a few definitions.

By a \emph{word} on a set $X$ we mean a finite string of symbols from $X$, i.e., an element of the free monoid $X^*$.  Suppose $X=(X,\prec)$ is totally ordered and let $w=w_1\cdots w_n \in X^*$. Recall that an \emph{inversion} in $w$ is a pair $(i,j)$ with $1 \leq i < j \leq n$ and $w_j \prec w_i$. The number of inversions in $w$ is denoted $\inv(w)$.  A \emph{descent} in $w$ is  an index $1 \leq i < n$ such that $w_{i+1} \prec w_{i}$.   Let $\Des w$ denote the set of descents in $w$. Then the \emph{descent number} and \emph{major index} of $w$ are  defined by $\des \w := |\Des w|$ and $\maj w := \sum_{i \in \Des w} i$, respectively.

We have elected to frame our results in terms of two different types of \emph{ascents} that can occur in words on the alphabet $X=\{1,\ldots,k\}$.  Here, and throughout, $k$ denotes a fixed but arbitrary positive integer. 

\begin{defn}
Let $\words{k}$ be the set of words on $\{1,\ldots,k\}$ that end with a single $k$.
\end{defn}

\newcommand{\rev}[1]{\sigma_{#1}}

\begin{defn}
\label{defn:ascents}
Let $\w = w_1 \cdots w_n \in \words{k}$. Then:
\begin{itemize}
\item An \emph{\aascent} of $\w$ is an index $1 \leq i < n$ such that  $\w_i < \w_{i+1}$. 
\item  A \emph{\bascent} of $\w$ is an index $1 \leq i < n$  such that either $\w_i < \w_{i+1}$ or $w_i=w_{i+1} < k$. 
\end{itemize}
The set of \aascent{s} of $\w$ is denoted  $\aAsc(\w)$. Its size and sum are denoted by $\aasc(\w):=|\aAsc(\w)|$ and $\amaj(\w) := \sum_{i \in \aAsc(\w)} i$.  The notations $\bAsc(\w)$, $\basc(\w)$ and $\bmaj(\w)$ are defined similarly.
\end{defn}

For example,  $w=322441314 \in \words{4}$ has  $\aAsc(\w)=\{3,6,8\}$ and $\bAsc(w)=\{2,3,6,8\}$, hence $\aasc(\w)=3$,  $\amaj(\w)=17$, $\basc(w)=4$, and $\bmaj(w)=19$.   We note that $\amaj(\w)$ is often called the \emph{comajor index} of $w$.

We can now fully describe the bijections mentioned earlier. Explicit constructions and examples  will be provided in Section~\ref{sec:encoding}.  

\begin{thm}
\label{thm:burge}
For each $k \geq 2$ there is a bijection $\map{\akburge}{\Pset}{\words{k}}$ such that if $w=\akburge(\l)$ and $a_i$ is the number of occurrences of $i$ in $\w$ then:
\begin{enumerate}
\item	$\l$ contains exactly $a_i$ parts congruent to $i$ modulo $k-1$, for $1 \leq i < k$
\item	$|\l| = (k-1)\amaj(\w) - \sum_{i=1}^{k-2} (k-1-i)a_i$
\item	$\km{\l} = \aasc(\w)$
\end{enumerate}
\end{thm}

\begin{thm}
\label{thm:bburge}
For each $k \geq 1$ there is a bijection $\map{\bkburge}{\Dset}{\words{k+1}}$  such that if $w=\bkburge(\l)$ and $a_i$ is the number of occurrences of $i$ in $\w$ then:
\begin{enumerate}
\item	$\l$ contains exactly $a_i$ parts  congruent to $i$ modulo $k$, for $1 \leq i \leq k$
\item	$|\l| = k\,\bmaj(\w) - \sum_{i=1}^{k-1} (k-i)a_i$
\item	$\km{\l} = \basc(\w)$
\end{enumerate}
\end{thm}

In case $k=2$,   Theorem~\ref{thm:burge} asserts a correspondence between partitions $\l \in \Pset$ and words $w \in (2^*1)^*2$ under which $|\l|=\amaj(w)$, $\mu_2(\l)=\aasc(w)$, and  $\len{\l}$ counts  occurrences of 1 in $w$.  After swapping the roles of 1 and 2 (which transforms $\amaj \rightarrow \maj$ and $\aasc \rightarrow \des$), this is exactly the correspondence described by Burge in~\cite[\S3]{Bur-1}.
See also~\cite{AndBres} for a reinterpretation in terms of lattice paths.

We pause for a brief comment on notation.
Evidently we could replace $(\amaj, \aasc)$ with the more usual pair $(\maj, \des)$ by globally swapping symbols $i$ and $k+1-i$.   We have  chosen the present indexing  because it simplifies our later analysis.

The promised refinements of Theorems~\ref{thm:acz1} and~\ref{thm:acz2} are stated below. The reader can verify that~\eqref{eq:acz1} and~\eqref{eq:acz2} are recovered from~\eqref{eq:main1} and~\eqref{eq:main2} by setting  $y_i=y$ for all $i$.

\begin{thm}
\label{thm:main}
For $\l \in \Pset$ and $m \geq 1$ let $\pcnt{m}(\l)=(c_1,\ldots,c_{m})$, where $c_i$ is the number of parts of $\l$ congruent to $i$ modulo $m$. Then for $k \geq 2$ we have
\begin{align}
 		\sum_{\l \in \Pset} z^{\km{\l}} q^{|\l|} \vec{y}^{\pcnt{k-1}(\l)} 	&= \qf{z}{q^{k-1}}{\infty}\sum_{n \geq 0} 
		\frac{z^n}{\qf{q^{k-1}}{q^{k-1}}{n} \prod_{i=1}^{k-1}\qf{y_i q^{i}}{q^{k-1}}{n}} \label{eq:main1} \\
		\shortintertext{and}
 		\sum_{\l \in \Dset} z^{\km{\l}} q^{|\l|} \vec{y}^{\pcnt{k}(\l)} 	&= \qf{z}{q^k}{\infty}\sum_{n \geq 0} 
		\frac{\prod_{i=1}^k \qf{-y_i q^i}{q^k}{n}}{\qf{q^k}{q^k}{n}} z^n,
		\label{eq:main2}
\end{align}
where $\vec{y}=(y_1,\ldots,y_{k-1})$ in the former and $\vec{y}=(y_1,\ldots,y_k)$ in the latter.
\end{thm}

The paper is organized as follows. We begin in Section~\ref{sec:encoding}  by constructing the  bijections $\akburge$ and $\bkburge$ described by Theorems~\ref{thm:burge} and~\ref{thm:bburge}.  We also introduce another bijection between partitions and $k$-ary words that is constructed in a similar manner but maps the size of a partition to the inversion number of the corresponding word. Section~\ref{sec:eulerian} reviews some classical ``$q$-enumeration'' and culminates in a proof of Theorem~\ref{thm:main} as a corollary of our bijections. 
The  remainder of the paper is more expository in nature.  In Section~\ref{sec:connections} we discuss the combinatorial connection between 2-measure and  Durfee squares that leads to Theorem~\ref{thm:abd}. Finally,  Section~\ref{sec:schur} describes how the encoding $\bkburge$ (Theorem~\ref{thm:bburge}) can be used to derive a partition identity of Schur and few variations.


\section{Partitions and $k$-ary Words}
\label{sec:encoding}


\subsection{Frequency Lists and Related Notation}

Let $\Fset$ be the set of all finitely supported sequences of nonnegative integers.  We will work  with partitions $\l \in \Pset$ exclusively via their \emph{frequency lists}  $(f_1,f_2,\ldots,) \in \Fset$, where $f_i$ is the number of parts of $\l$ equal to $i$.   For example, $(5,3,3,1,1,1) \in \Pset$ corresponds with $(3,0,2,0,1,0,0,\ldots) \in \Fset$, and $\varepsilon \in \Pset$ corresponds with $\emptyf=(0,0,\ldots)$. 
This identification of $\Fset$ and $\Pset$ should be kept in mind throughout.

We extend notation from $\Pset$ to $\Fset$ in the obvious way, letting 
$$
	\textstyle
	|f| :=  |P(f)| = \sum_i if_i, \qquad\qquad
	\len{f} := \len{P(f)} = \sum_i f_i
$$
and $\km{f}:=\km{P(f)}$.   Thus $\km{f}$ is the maximum size of a set $I$ of positive integers such that $f_i > 0$ for all $i \in I$, and $|i-j| \geq k$ for all distinct $i,j \in I$.

\begin{convention} When working with generic $f \in \Fset$, we adopt the convention that $f_i=0$ for $i < 1$.  When working with specific $f \in \Fset$,  trailing zeroes will often be omitted; e.g. $(0,3,2,0,1) = (0,3,2,0,1,0,0,\ldots)$.
\end{convention}

Let $f \in \Fset$ and $k\geq 1$. Define integers $i_0 < i_1 < \cdots < i_m$ inductively as follows:  Begin with $i_0=-\infty$ and then, for $j \geq 1$, let $i_j$ be the least positive integer such that $i_j \geq i_{j-1} + k$ and $f_{i_j} > 0$, stopping when no such index exists.  Let $\Ilk(f) := \{i_1,\ldots,i_m\}$.  Thus the elements of $\Ilk(f)$ index the left ends of the $k$-tuples obtained by parsing $f$ from left-to-right and greedily selecting $(f_i,f_{i+1},\ldots,f_{i+k-1})$ with $f_i > 0$. We call the tuples $(i_j,\ldots,i_j+k-1)$ the \emph{forward clusters} of $f$.

Similarly define $\Irk(f) := \{i_1,\ldots,i_m\}$ by setting $i_0=\infty$ and letting $i_j$ be the greatest positive integer such that $i_j \leq i_{j-1}-k$ and $f_{i_j} > 0$, stopping when no such index exists.   Thus $i_1 > i_2 > \cdots > i_{m}$ index the right ends of the $k$-tuples identified by parsing $f$ from right-to-left and greedily selecting $(f_{i-k+1},\ldots,f_i)$ with $f_i > 0$.  The tuples $(i_j-k+1,\ldots,i_j)$ are  the \emph{backward clusters} of $f$.   Note that this construction relies on the convention that $f_j = 0$ for $j < 0$. 

\newcommand{\pos}[2]{\underset{\rule{0pt}{2ex}#1}{#2}}
\begin{exmp}
\label{exmp:scanning}
Let $f = (2,2,0,3,1,0,0,4,0,2,1,1) \in \Fset$. Then 
\begin{xalignat*}{2}
\Ilk[2](f)&=\{1,4,8,10,12\} & \Irk[2](f)&=\{2,5,8,10,12\} \\
\Ilk[3](f)&=\{1,4,8,11\} & \Irk[3](f)&=\{2,5,8,12\} \\
\Ilk[4](f)&=\{1,5,10\} & \Irk[4](f) &= \{4,8,12\}
\end{xalignat*}
The computation of $\Ilk[3](f)$ and $\Irk[3](f)$ are illustrated below, with forward/backward clusters indicated by arrows and the corresponding indices annotated below.
$$
(\Lpair{\pos{1}{2},2,0},\Lpair{\pos{4}{3},1,0},0,\Lpair{\pos{8}{4},0,2},\Lpair{\pos{11}{1},1,0},0,\ldots)
\qquad\quad
\Rpair{\ \ (2,\pos{2}{2}},\Rpair{0,3,\pos{5}{1}},\Rpair{0,0,\pos{8}{4}},0,\Rpair{2,1,\pos{12}{1}},0,\ldots)
$$
\end{exmp}

By construction it is clear that the elements of both $\Ilk(f)$ and $\Irk(f)$ differ pairwise by at least $k$.
 Moreover, if $f_i > 0$ then we have $l \leq i < l+k$ and $r-k < i \leq r$ for unique $l \in \Ilk(f)$ and $r \in \Irk(f)$.  That is, the index of every positive entry of $f$ lies within exactly one forward  and one backward cluster.  

\begin{prop}
\label{prop:kmcount}
For any $f \in \Fset$ we have $\km{f} = |\Ilk(f)| = |\Irk(f)|$.
\end{prop}

\begin{proof}
Suppose $s_1 < s_2 < \cdots < s_r$ satisfies $f_{s_j}>0$  and $s_{j+1}-s_j \geq k$ for all $j$.  As noted above, each $s_j$ is contained in a unique forward cluster, and since the clusters are of length $k$ none can contain more than one $s_j$.  So we have $r \leq |\Ilk(f)|$ and thus $\km{f} \leq |\Ilk(f)|$. But if $\Ilk(f) = \{i_1 < i_2 < \cdots < i_m\}$ then  $f_{i_j} > 0$ and $i_{j+1}-i_j \leq k$ for all $j$, and thus $\km{f} \geq |\Ilk(f)|$.  So $\km{f}=|\Ilk(f)|$. The treatment of $\Irk(f)$ is identical. 
\end{proof}

Finally, for integer $j$ we define the promotion and demotion operators $\up{j}$ and $\down{j}$ on $\Fset$ as follows:   
$$
\up{j}(f)_i :=
\begin{cases}
f_i+1 &\text{if $i=j$} \\
f_i &\text{otherwise}.
\end{cases},
\qquad\quad
\down{j}(f)_i :=
\begin{cases}
f_i-1 &\text{if $i=j$} \\
f_i &\text{otherwise}.
\end{cases}
$$
Note that both $\up{j}$ and $\down{j}$ act as the identity on $\Fset$ when $j \leq 0$. 

\subsection{A Generalized Burge Encoding}
\label{sec:burgeconst}

We shall now describe a bijection between partitions and  words on an alphabet of size $k \geq 2$. This is a natural generalization of the correspondence between partitions and binary words described by Burge in~\cite[Section 3]{Bur-1}. 

\begin{convention}
Throughout this section,  unless otherwise noted, $k$ denotes a \emph{fixed} positive integer at least 2.
All constructions that follow depend on $k$, but we shall often leave this dependence implicit.
\end{convention}

We begin by defining some fundamental mappings on $\Fset$. 
First let $\map{\astar}{\Fset}{\Fset}$ and  $\map{\pt}{\Fset}{\Fset}$  be given by
$$
	\astar(f) := \Big(\prod_{i \in \Ilk(f)} \down{i}\up{i+k-1} \Big)(f),
		\qquad\qquad
	\pt(f) :=\Big(\prod_{i \in \Irk(f)} \up{i-k+1}\down{i}\Big)(f),
$$ 
In other words, $\astar$ performs the transformation
$$
	(f_i,\ldots,f_{i+k-1}) \mapsto (f_i-1,f_{i+1},\ldots,f_{i+k-2},f_{i+k-1}+1).
$$
for each forward cluster $(i,\ldots,i+k-1)$ of $f$, while $\pt$ performs the reverse transformation 
$$
	(f_{i-k+1},\ldots,f_{i}) \mapsto (f_{i-k+1}+1,f_{i-k},\ldots,f_{i-1},f_{i}-1)
$$
for each backward cluster.  Now  define $\map{\a_j}{\Fset}{\Fset}$ for $1 \leq j < k$ by
\begin{align*} 
  \a_j(f) &=
    \up{j}(f_1,\ldots,f_{j}, \astar(f_{j+1},f_{j+2},\ldots) ) \\
  &=(f_1,\ldots,f_{j-1},f_{j}+1,\astar(f_{j+1},f_{j+2},\ldots) ).
\end{align*}

It is clear from the definitions that the forward clusters of $f$ are precisely the backward clusters of $\astar(f)$. From this it follows that $\pt$ is a left-inverse for each $\a_j$.   The full extent of this inverse relationship will be made precise below.  

We emphasize that all of these mappings depend on the ambient  parameter $k$. For instance, $\a_3$ is undefined when $k=2$ and has very different meanings for $k=3$ versus $k=4$. As such we should rightly use notation such as $\a^{(k)}_j$  and $\pt^{(k)}$, but have avoided this practice because it becomes unwieldy.

\begin{exmp}
\label{exmp:kburge}
Suppose $k=3$ and let $f = (2,2,1,3,1,0,0,4,0,2,0,1)$, so that $\Ilk[3](f)=\{1,4,8,12\}$ and $\Irk[3](f)=\{2, 5, 8, 12\}$.  The computation of $\a_j(f)$ and $\pt(f)$ is illustrated below.  The arrows arch over the clusters indexed by $\Ilk[3](f)$ and $\Irk[3](f)$, and one ``unit'' is pushed along each arrow from tail to tip. 
\begin{align*}
	\a_3(f)
&=\a_3(\Lpair{2,2,1},\Lpair{3,1,0},0,\Lpair{4,0,2},0, \Lpair{1,0,0}) \\
&=(1,2,2,2,1,1,0,3,0,3,0,0,0,1)  \\[1.5ex]
	\a_1(f) &= (2+1,\a_3(\Lpair{2,1,3},\Lpair{1,0,0},\Lpair{4,0,2},0,\Lpair{1,0,0})) \\
	 &= (3, 1,1,4,0,0,1,3,0,3,0,0,0,1) \\[1.5ex]
	\a_2(f) &= (2,2+1,\a_3(\Lpair{1,3,1},0,0,\Lpair{4,0,2},0,\Lpair{1,0,0})) \\
	 &= (2,3,0,3,2,0,0,3,0,3,0,0,0,1) \\[2ex] 
	 \pt(f)&=
	\Rpair{\pt(2,2},\Rpair{1,3,1},   		
	\Rpair{0,0,4},0,\Rpair{2,0,1})  \\
	&= (2,1,2,3,0,1,0,3,0,3).
\end{align*}
The reader can verify that $\a_j(f) \in \Aset_j$  and $\pt(\a_j(f))=f$ for each $j \in \{1,2,3\}$, whereas $\a_j(\pt(f))=f$ only for $j=2$.  
\end{exmp}

\begin{defn}
For $1 \leq j < k$, let $\Aset_j$ be the set of all $f \in \Fset$ such that $j \in  \Irk(f)$ (and therefore $j=\min \Irk(f)$). Note that these sets are disjoint.  Let $\Astar$ be the complement of 
$\bigcup_{j=1}^{k-1} \Aset_{j}$ in $\Fset$. Thus $\Astar$ contains $f=\emptyf$ along with all $f$ such that $\min \Irk(f) \geq k$.
\end{defn}

\begin{prop}
\label{prop:bijections}
For $1 \leq j \leq k$, the map $\a_j$ defined above is a bijection from $\Fset$ to $\Aset_j$.   The map $\pt$ is a $k$-to-$1$ surjection from $\Fset$ to $\Fset$ that restricts to $\a_j^{-1}$ on $\Aset_j$.
\end{prop}

\begin{proof}

As noted above, the forward clusters of $f$ are  the backward clusters of $\astar(f)$.  That is,  
\begin{align}
\label{eq:irk_astar}
	\Irk(\astar(f))=\Ilk(f)+k-1,
\end{align}
where $S+j := \{s + j \,:\, s \in S\}$ denotes the translation of set $S$ by $j$. Therefore $\astar(f) \in \Astar$ and $\pt(\astar(f))=f$ for all $f$.  Similarly, if $f \in \Astar$ then the forward clusters of $\pt(f)$ match the backward clusters of $f$, and thus $\astar(\pt(f))=f$. 

Now suppose $1 \leq j < k$.  Then
\begin{align}
\Irk(\a_j(f)) &=
\Irk(f_1,\ldots,f_{j-1},f_j+1,0,0,\ldots) \cup \Irk(0,\ldots,0,\astar(f_{j+1},f_{j+2},\ldots)) \notag \\
&= \{j\} \cup (\Irk(\astar(f_{j+1},f_{j+2},\ldots))+j) \label{eq:irk_aj} \\
&= \{j\} \cup (\Ilk(f_{j+1},f_{j+2},\ldots)+k-1+j). \notag
\end{align}
Therefore $\a_j(f) \in \Aset_j$ for all $f \in \Fset$ and we have
\begin{align*}
\pt(\a_j(f)) &= \pt(f_1,\ldots,f_{j-1},f_j+1,\astar(f_{j+1},f_{j+2},\ldots)) 
\\
&= (f_1,\ldots,f_{j-1},f_j,\pt(\astar(f_{j+1},f_{j+2},\ldots))) \\
&= f.
\end{align*}
Conversely if $f \in \Aset_j$ then $j = \min \Irk(f)$ and  $(f_{j+1},f_{j+2},\ldots) \in \Astar$, which implies
\begin{align*}
\a_j(\pt(f))
&=\a_j(f_1,\ldots,f_{j-1},f_j-1,\pt(f_{j+1},f_{j+2},\ldots)) \\
&=(f_1,\ldots,f_j,\a_k(\pt(f_{j+1},f_{j+2},\ldots))) \\
&=f.
\end{align*}
\end{proof}

From hereon we shall simplify notation and write $\pt f$ in place of $\pt(f)$.  From the definition of $\pt$ it is clear that $0 \leq |\pt f| < |f|$ for $f \neq \emptyf$.  Thus for any $f$ there exists a least positive integer $r$  such that $\pti{f}{r-1} = \emptyf$. Together with Proposition~\ref{prop:bijections} this permits a canonical decomposition of elements of $\Fset$.

\begin{defn}
Let $f \in \Fset$ and let $r$ be the least positive integer such that $\pti{f}{r-1}=\emptyf$. 
We define the \emph{\akcode{k}} of $f$ to be the word $\akburge(f) := \w_1\cdots\w_{r}$ on symbols $\{1,2,\ldots,k\}$ determined by the rule $\w_{i} = j \Leftrightarrow \pti{f}{i-1} \in \Aset_j$.
\end{defn}

\begin{exmp}
\label{exmp:burge}
Let $f=(1,2,1,0,0,1,0,0,0,0,1)$.  The computations of $\akburge[3](f)=312232113$ and $\akburge[4](f)=2313224$ are illustrated below, at left and right, respectively.  At each nonterminal step, $\w_{i}$ is determined from the least element of $\Irk(\pti{f}{i-1})$.  
%
$$
\begin{array}[t]{c|l|l|l}
i & \pti{f}{i-1} & \Irk[3](\pti{f}{i-1}) & \w_{i} \\
\hline
1  & (1,2,1,0,0,1,0,0,0,0,1) & \{3,6,11\} & 3 \\
2  & (2,2,0,1,0,0,0,0,1) & \{{1},4,9\} & 1 \\
3 & (1,3,0,0,0,0,1) & \{2,7\} & 2 \\
4 & (1,2,0,0,1) & \{2,5\} & 2 \\
5 & (1,1,1) & \{3\} & 3 \\
6 & (2,1,0) & \{2\} & 2 \\
7 & (2) & \{1\} & 1 \\
8 & (1) & \{1\} & 1 \\
9 & \emptyf &  & 3
\end{array}
\qquad
\begin{array}[t]{c|l|l|l}
i & \pti{f}{i-1} & \Irk[4](\pti{f}{i-1}) & \w_{i} \\
\hline
1  &(1,2,1,0,0,1,0,0,0,0,1) & \{2,6,11\} & 2 \\
2  &(1,1,2,0,0,0,0,1) & \{3,8\} & 3 \\
3 & (1,1,1,0,1) & \{1,5\} & 1 \\
4 & (0,2,1) & \{3\} & 3 \\
5 & (0,2) & \{2\} & 2 \\
6 & (0,1) & \{2\} & 2 \\
7 & \emptyf &  & 4
\end{array}
$$

\end{exmp}

Let $e^j =(0,\ldots,0,1,0,\ldots) \in \Fset$ be the sequence with a 1 in the $j$-th position and zeroes elsewhere.
From the definition of $\pt$ it is clear that $\pt f = \emptyf$ if and only if either $f = \emptyf$ or $f=e^j$ for $1 \leq j < k$. Since $\emptyf \in \Astar$ and $e^j \in \Aset_j$, it follows that either $\akburge(f)=k$, in which case $f=\emptyf$,  or the final two symbols of $\akburge(f)$ are $jk$ for some $j \neq k$.  In particular, $\akburge(f)$ always ends in a single $k$.


We have just observed that the encoding $f \mapsto \akburge(f)$ maps $\Fset$ into $\words{k}$, the set of words on $\{1,\ldots,k\}$ that end in a single $k$. From Proposition~\ref{prop:bijections} it is easy to see that this is in fact a bijection.  Indeed, since $\pt$ restricts to $\a_j^{-1}$ on $\Aset_j$, the construction of $\akburge(f)$   ensures that $\pti{f}{i-1}= (\a_{\w_i} \circ \pt)(\pti{f}{i-1})=\a_{\w_i}(\pti{f}{i})$ for all $i$. We therefore have
$$
f =\a_{\w_1}(\pti{f}{1})=\a_{\w_1}\a_{\w_2}(\pti{f}{2})=\a_{\w_1}\a_{\w_2}\a_{\w_3}(\pti{f}{3})=\cdots=(\a_{\w_1}\a_{\w_2}\cdots\a_{\w_r})(0),
$$
or, more succinctly,  $f=\a_{\w}(0)$.
Observe that this identity can be unambiguously unwound to determine all the $\w_i$, since $f=\a_j(g)$  implies  $f \in \Aset_j$ and $\pt(f)=g$.   Thus the maps $f \mapsto \akburge(f)$ and $w \mapsto \a_w(0)$ are mutually inverse. The right-to-left reading of $\akburge(f)$ specifies the unique manner in which $f$ can be constructed from $\emptyf$ by iterated applications of the maps $\{\a_1,\ldots,\a_k\}$,  beginning with a single application of $\a_k$. (Note that the first application of $\a_k$ is inert.)

\begin{exmp}
Consider the \akcode{3}  $\w=131223$. The corresponding frequency list $f=\a_{\w}(0)$ can be found as follows: 
$$
\emptyf \overset{\a_3}{\mapsto} 
\emptyf \overset{\a_2}{\mapsto} 
(0,1) \overset{\a_2}{\mapsto} 
(0,2) \overset{\a_1}{\mapsto} 
(1,1,0,1) \overset{\a_3}{\mapsto}
(0,1,1,0,0,1)  \overset{\a_1}{\mapsto} (1,0,1,1,0,0,0,1)=f.$$
\end{exmp}

Let us now consider the behaviour of the statistics  $\km{\cdot}$ and $|\cdot|$ under the action of $\pt$, with the goal of mapping them through $\aburge{k}$  to natural statistics on $\words{k}$.

\begin{lem}
\label{lem:stats}
Suppose $f \in \Aset_i$. Then
\begin{enumerate}
\item	$\displaystyle
		\km{\pt f} = \begin{cases}
		\km{f}-1 &\text{if $\pt f \in \Aset_j$ with $i < j$} \\
		\km{f} &\text{otherwise}.
		\end{cases}
		$
\item	$\displaystyle
		|\pt f| = 
		\begin{cases}
		|f|-(k-1)\km{f} &\text{if $i=k$} \\
		|f|-(k-1)\km{f}+ (k-1 -i) &\text{otherwise}. \\
		\end{cases}
		$
\end{enumerate}
\end{lem}

\begin{proof}

For brevity let $g=\pt f$ and note that $f=\a_i(g)$ since $f \in \Aset_i$.  

First consider the case $i=k$.  Here~\eqref{eq:irk_astar} gives $\Irk(f) = \Ilk(g)+k-1$, and so Proposition~\ref{prop:kmcount} yields $\km{g} = \km{f}$.
Moreover $|f|=|g| + (k-1)\km{g}$, since $\astar$ transforms $g$ into $f$ by promoting $|\Irk(g)|=\km{g}=\km{f}$  clusters, and each promotion increases $|g|$ by $k-1$.  This completes the proof of Claims 1 and 2 in case $i=k$.

Now suppose $1 \leq i < k$. In this case~\eqref{eq:irk_aj}  gives
$$
	\Irk(f) = \{i\} \cup (\Irk(g_{i+1},g_{i+2},\ldots)+i)
$$
and therefore 
\begin{equation}
\label{eq:kmid}
	\km{f}=1+\km{g_{i+1},g_{i+2},\ldots}.
\end{equation}
Observe that since $g \in \Aset_j$ we have
\begin{align}
 i < j < k &\implies
 \Irk(g) = \Irk(g_{i+1},g_{i+2},\ldots)+i \notag \\
 &\implies
 \km{g}=\km{g_{i+1},g_{i+2},\ldots}, \label{eq:kmid1}
\shortintertext{and}
 j \leq i < k &\implies
 \Irk(g) = \{j\} \cup (\Irk(g_{i+1},g_{i+2},\ldots)+i) \notag \\
&\implies
 \km{g}=\km{g_{i+1},g_{i+2},\ldots}+1. \label{eq:kmid2}
\end{align}
Furthermore, by definition of $\a_i$ we have 
\begin{equation}
\label{eq:kmid3}
|f|=|g| + i + (k-1)\km{g_{i+1},g_{i+2},\ldots}.
\end{equation} 
Identities \eqref{eq:kmid}--\eqref{eq:kmid3} together establish Claims 2 and 3 in case $i < k$. 
\end{proof}

We are now prepared to  prove the main result of this section, which obviously implies Theorem~\ref{thm:burge} when translated from $\Fset$ to $\Pset$. 
As noted in the Introduction, the case $k=2$ was described by Burge~\cite[Section 3]{Bur-1}.

\begin{thm}
\label{thm:Fburge}
The encoding $f \mapsto \akburge(f)$ is a bijection from $\Fset$ to $\words{k}$.  Moreover, if $\w = \akburge(f)$ and  $a_i$ is the number of occurrences of $i$ in $\w$, then:
\begin{enumerate}
\item	$a_i = \sum_{j \equiv i \pmod{k-1}} f_j$, for $1 \leq i < k$
\item	$\len{f} = a_1 + \cdots + a_{k-1}$
\item	$\km{f} = \aasc(\w)$
\item	$|f| =(k-1)\amaj(\w) - \sum_{i=1}^{k-1} (k-1-i)a_i$
\end{enumerate}
\end{thm}

\begin{proof}
We have already established the bijectivity of the encoding.

Observe that the sum $\sum_{j \equiv i \pmod{k-1}} f_j$ is invariant under $\down{s}\up{s+k-1}$ for all $s \geq 1$ and also under $\up{s}$ for $s \not \equiv i \pmod{k-1}$.  It is therefore invariant under $\a_t$ for $t \neq i$ and increases by 1 under $\a_i$. Claim 1 follows from the fact that $f=\a_w(0)$. Claim 2 is  immediate from Claim 1.  

Claim 3 is clear from iterative application of Lemma~\ref{lem:stats}(1).

We now address Claim 4. 
Suppose $\w = \w_1 \cdots \w_{n+1}$ and let $\aAsc(\w) = \{i_1, i_2, \ldots, i_m \}$, where  $1 \leq i_1 < i_2 < \cdots < i_m=n$.  Then iterating Lemma~\ref{lem:stats}(1) gives  $\km{\pti{f}{r}}=m-j$ for $r$ in the range $i_j \leq r < i_{j+1}$, where we take $i_0=0$ and $i_{m+1}=\infty$. Hence 
$$
\sum_{r=1}^n \km{\pti{f}{r}} = \sum_{j=0}^{m-1} (m-j)(i_{j+1}-i_j) = \sum_{j=0}^{m-1} i_{j+1} = \amaj(\w).
$$	
But iterating Lemma~\ref{lem:stats}(2) yields
\begin{align*}
	|f| &=(k-1)\sum_{r=1}^{n} \km{\pti{f}{r}} + \sum_{i=1}^{k-1} a_i(i-k+1),
 \end{align*}
and Claim 4 follows.
\end{proof}

\begin{exmp}
Let $f=(1,2,1,0,0,1,0,0,0,0,1)$, which corresponds to the partition $\l=(11,6,3,2,2,1)$ of $25$.  
In Example~\ref{exmp:burge} we determined the \akcode{3} of $f$ to be $
\w=312232113.$
Thus $a_1=a_2=3$, and we note that $\l$ has 3 odd and 3 even parts, in agreement with (1) and (2) of Theorem~\ref{thm:burge}.  Furthermore we have $\aAsc(\w)=\{2,4,8\}$, so $\aasc(\w)=3$ and $\amaj(\w)=2+4+8=14$. Claims (3) and (4) of the Theorem correctly predict $\mu_3(f)=3$ and $|f|=2(14)-(1)a_1-(0)a_2=25$.
\end{exmp}

\begin{exmp}
The following table lists all partitions of 5 along with their corresponding frequency lists and \akcode{4}s. Each of these partitions has $4$-measure equal to 1, reflected by a singular \aascent in each codeword (occurring at the second-to-last index). Observe that $|\l| = 3\amaj(\w) - 2a_1-a_2$ and $\len{\l}=a_1+a_2+a_3$ in every row of the table, in accordance with Theorem~\ref{thm:burge}.
$$
\begin{array}{c|c|r|c|c|c|c}
\l & f & w=\aburge{4}(f) &  \amaj(\w) & a_1 & a_2 & a_3  \\
\hline
(5) & (0,0,0,0,1) & 424 & 2 & 0 & 1 & 0\\
(4,1) & (1,0,0,1) & 4114  & 3 & 2 & 0 & 0\\
(3,2) & (0,1,1) & 324  & 2 & 0 & 1 & 1 \\
(3,1,1) & (2,0,1) & 3114  & 3 & 2 & 0 & 1 \\
(2,2,1) & (1,2) & 2214  & 3 & 1 & 2 & 0\\
(2,1,1,1) & (3,1) & 21114 & 4 & 3 & 1 & 0 \\
(1,1,1,1,1) & (5) & 111114  & 5 & 5 & 0 & 0
\end{array}
$$
\end{exmp}

\subsection{Encoding Partitions with Distinct Parts}

\newcommand{\bmax}{\b_{k+1}}
\newcommand{\dt}{\partial}
\newcommand{\BFset}{\Fset^{01}}
\newcommand{\BIlk}{\Ilk}
\newcommand{\BIrk}{\Irk}

We now turn to study  partitions $\l \in \Dset$ with distinct parts. Evidently these are the partitions whose frequency lists have entries 0 or 1. Let $\BFset$ denote the set of all such binary sequences in $\Fset$.  

We seek an encoding of $\BFset$  akin to Theorem~\ref{thm:Fburge}. However, since our prior operators $\a_1,\ldots,\a_k$  evidently do not map $\BFset$ into itself,  we must modify these  transformations so they preserve binary sequences.  

\begin{convention}
Unless otherwise noted, $k \geq 1$ is a fixed ambient parameter.
\end{convention}

We begin by defining the mappings $\b_1,\ldots,\b_{k+1}$  on $\BFset$ as follows.  Let 
$$
	\bstar(f):=\Big(\prod_{i \in \BIlk(f)} \down{i}\up{i+k}\Big)(f),
$$ 
and then, for $1 \leq j \leq k$, let
\begin{align*} 
  \b_j(f) 
  &=\up{j}(f_1,\ldots,f_{j-1},\bstar(f_{j},f_{j+1},\ldots) ) \\
  &=(f_1,\ldots, f_{j-1},\up{1}\bstar(f_{j},f_{j+1},\ldots)).
\end{align*}
Let us now recycle notation from the previous section and define the inverse map  $\map{\dt}{\BFset}{\BFset}$ by
$$
	\dt(f) :=\Big(\prod_{i \in \BIrk(f)} \up{i-k}\down{i}\Big)(f).
$$

We draw the reader's attention to the subtle difference in the definition of $\bstar$ versus $\astar$.  Both act on forward $k$-clusters $(i,\ldots,i+k-1)$ of $f$, but $\bstar$ pushes one unit from $f_i$ to $f_{i+k}$ rather than from $f_i$ to $f_{i+k-1}$.  Note that $\bstar$ (and thus each $\b_i$) does indeed preserve binary  sequences.  This is a consequence of the following observation: If $f \in \BFset$ and  $i\in\Ilk(f)$, then $f_{i+k}>0$ if and only if $i+k\in\Ilk(f)$.
Analogous statements hold for $\dt$.

\begin{defn}
For $1 \leq j \leq k$, let $\Bset_j$ be the set of  $f \in \BFset$ such that $j=\min \Irk(f)$.  Let $\Bstar$ be the complement of 
$\bigcup_{j=1}^{k} \Bset_{j}$ in $\BFset$. Thus $\Bstar$ contains $f=\emptyf$ along with all $f$ such that $\min \Irk(f) > k$.
\end{defn}

\begin{exmp}
\label{exmp:bkburge}
Let $k=3$ and $f = (1,0,1,1,0,1,0,0,1,1,0,0,1,0,1,1,0,1)$. Then $\Ilk[3](f)=\{1,4,9,13,16\}$ and $\Irk[3](f)=\{3, 5,10, 15, 18\}$, so $f \in \Bset_3$.  The computation of $\b_j(f)$ and $\dt(f)$ is illustrated below.  Arrows arch over $(k+1)$-tuples $(f_i,\ldots,f_{i+k})$ such that $i \in \Ilk[3](f)$ (in the case of $\b_j$) or  $i+k \in \Irk[3](f)$ (in the case of $\pt$). One unit is pushed along each arrow from tail to tip.
\begin{align*}
	\b_4(f)
&=\b_4(\BLpair{1,0,1},\BLpair{1,0,1},0,0,\BLpair{1,1,0},0,\BLpair{1,0,1},\BLpair{1,0,1},0,0,\ldots) \\
&=(0,0,1,1,0,1,1,0,0,1,0,1,0,0,1,1,0,1,1) \\[1.5ex]
	\b_1(f) &= 
	\up{1}(
	\b_4(\BLpair{1,0,1},\BLpair{1,0,1},0,0,\BLpair{1,1,0},0,\BLpair{1,0,1},\BLpair{1,0,1},0,0,\ldots)) \\
	 &= \up{1}(0,0,1,1,0,1,1,0,0,1,0,1,0,0,1,1,0,1,1,0,\ldots) \\
	 &= (1,0,1,1,0,1,1,0,0,1,0,1,0,0,1,1,0,1,1,0,\ldots) \\[1.5ex]
	\b_2(f) &= 
	\up{2}(
	1,\b_4(0,\BLpair{1,1,0},\BLpair{1,0,0},\BLpair{1,1,0},0,\BLpair{1,0,1},\BLpair{1,0,1},0,0,\ldots)) \\
	 &= (1,1,0,1,0,1,0,0,1,1,0,1,0,0,1,1,0,1,0,\ldots)  \\[2ex] 
	\b_3(f) &= 
	\up{3}(
	1,0,\b_4(\BLpair{1,1,0},\BLpair{1,0,0},\BLpair{1,1,0},0,\BLpair{1,0,1},\BLpair{1,0,1},0,0,\ldots)) \\
	 &= (1,0,1,1,0,1,0,0,1,1,0,1,0,0,1,1,0,1,0,\ldots)  \\[2ex] 
	 \dt(f)&=
	\dt\BRpair{(1,0,1},\BRpair{1,0,1},0,\BRpair{0,1,1},0,0,\BRpair{1,0,1},\BRpair{1,0,1},0,\ldots)  \\
	&= (1,0,1,1,0,0,1,0,1,0,0,1,1,0,1,1,0,0,\ldots).
\end{align*}
The reader can verify that $\b_j(f) \in \Bset_j$  and $\pt(\b_j(f))=f$ for all $j$, while $\b_j(\dt(f))=f$ only for $j=3$.  
\end{exmp}

Note that a binary sequence $f \in \BFset$ can  be identified with its support $S(f):=\{i \,:\, f_i = 1\}$. This amounts to viewing the corresponding partition $\l \in \Dset$ simply as a set of positive integers.   The action of $\bstar$   can then be described as follows:  First scan $S(f)$ from least-to-greatest to greedily identify maximal arithmetic sequences having common difference $k$. The union of these sequences is $\Ilk(f)$.  Then $\bstar$ shifts each such sequence up by $k$ and leaves the remaining elements of $S$ unchanged.  That is, 
\begin{equation}
\label{eq:support}
S(\bstar(f)) = (S(f) \setminus \Ilk(f)) \cup (\Ilk(f) + k).
\end{equation}
The action of $\dt$ can be described similarly: Parse $S(f)$ from greatest-to-least to identify maximal sequences of common difference $-k$, then shift each down by $k$ and discard anything $\leq 0$.

For example, let $\l=(20,19,16,12,10,8,7,5,3,2)$ and $k=3$. Scanning from least-to-greatest yields maximal sequences $(2,5,8)$, $(12)$ and $(16,19)$, which $\b_4$ transforms into $(5,8,11)$, $(15)$ and $(19,22)$.  Thus $\b_4(\l)=(22,20,19,15,11,10,8,7,5,3)$.  Parsing $\l$ from greatest-to-least identifies sequences $(20), (16), (12), (8,5,2)$, which are transformed by $\dt$ into $(17), (13), (9), (5,2)$, resulting in $\dt(\l)=(19,17,13,10,9,7,5,3,2)$.


\begin{prop}
\label{prop:Bbijections}
For $1 \leq j \leq k+1$, the map $\b_j$ defined above is a bijection from $\BFset$ to $\Bset_j$.   The map $\dt$ is a $(k+1)$-to-$1$ surjection from $\BFset$ to $\BFset$ that restricts to $\b_j^{-1}$ on $\Bset_j$.
\end{prop}

\begin{proof}

All claims hinge on the following identities:
\begin{align}
\label{eq:irk_bstar}
\Irk(\b_{k+1}(f))&=\Ilk(f)+k, \quad\text{for  $f \in \BFset$} \\
\intertext{and}
\label{eq:irk_dt}
\Ilk(\dt(f))&=\Irk(f)-k, \quad\text{for  $f \in \Bstar$}.
\end{align}
 With these in hand, the proof is essentially identical to that of Proposition~\ref{prop:bijections}. We carefully prove~\eqref{eq:irk_bstar}. The proof of~\eqref{eq:irk_dt} is similar and omitted, and we direct the reader to Proposition~\ref{prop:bijections} for the rest.

Let $f^*=\bstar(f)$ and suppose $\Irk(f^*)=\{j_1,\ldots,j_m\}$ with $j_1 < \cdots < j_m$. We start by showing $\Irk(f^*) \subseteq \Ilk(f)+k$. Assume not and let $r$ be maximal such that $j_r-k \not \in \Ilk(f)$.
From the definition of $\bstar$ it is clear that $f^*_j=1$ if and only if either 
\begin{itemize}
 \item $f_j=1$ and $j\not\in \Ilk(f)$, or
 \item $j-k \in \Ilk(f)$.   
\end{itemize}
(Note that this is precisely~\eqref{eq:support}.)
In particular we  must have $f_{j_r}=1$ and $j_r \not \in \Ilk(f)$. 
By definition of $\Ilk(f)$, this implies the existence of $s \in \Ilk(f)$  in the interval $j_r-k < s < j_r$.  Also note that $r < m$, as  it is easy to see that $f_{j_m} = 0$.  Then maximality of $r$ gives $j_{r+1}-k \in \Ilk(f)$, and since elements of $\Ilk(f)$  differ by at least $k$ we  have $(j_{r+1}-k)-s \geq k$.  Altogether we have $j_r < s+k \leq j_{r+1}-k$, and this forces  $f^*_{s+k}=0$ by definition of $\Irk(f^*)$. This is contradictory, since $s \in \Ilk(f)$ implies $f^*_{s+k}=1$.

We now show $\Irk(f^*) \supseteq \Ilk(f)+k$. Suppose to the contrary that $i-k \in \Ilk(f)$ and  $i \not \in \Irk(f^*)$. Since $i-k \in \Ilk(f)$ we have $f^*_{i} =1$,  which together with $i \not \in \Irk(f^*)$ implies the existence of $s \in \Irk(f^*)$ with $i < s < i+k$.  Since $\Irk(f^*) \subseteq \Ilk(f)+k$,  we  have $s-k \in \Ilk(f)$. Therefore there are two elements of $\Ilk(f)$, namely $i-k$ and $s-k$, that differ by $(s-k)-(i-k) =s-i < k$, a contradiction.

\end{proof}

As in the previous section, we clearly have $0 \leq |\dt f| < |f|$ for $f \neq \emptyf$.  This permits the following definition of the \bkcode{k} of $f \in \BFset$,  which closely mimics our construction of its \akcode{k}.

\begin{defn}
Let $f \in \BFset$ and let $r$ be the least positive integer such that $\dti{f}{r-1}=\emptyf$. 
The \bkcode{k} of $f$ is the word $\bkburge(f) := \w_1\cdots\w_{r}$ on $\{1,2,\ldots,k+1\}$ determined by $\w_{i} = j \Leftrightarrow \dti{f}{i-1} \in \Bset_j$.
\end{defn}

\begin{exmp}
\label{exmp:bburge}
Computation of the \bkcode{3} of $\l=(22,18,16,14,13,11,9,5,2) \in \Dset$ is  illustrated below.  The result is 
$\bkburge(\l)=22434442121314$. 
$$
\begin{array}{c|l|l|l}
i & \pti{f}{i-1} & \Irk[3](\pti{f}{i-1}) & \w_{i} \\
\hline
1 & (0,1,0,0,1,0,0,0,1,0,1,0,1,1,0,1,0,1,0,0,0,1) & \{2,5,11,14,18,22\} & 2 \\
2 & (0,1,0,0,0,0,0,1,1,0,1,0,1,0,1,1,0,0,1) & \{2,9,13,16,19\} & 2 \\
3 & (0,0,0,0,0,1,0,1,0,1,1,0,1,0,1,1) & \{6,10,13,16\} & 4 \\
4 & (0,0,1,0,0,0,1,1,0,1,1,0,1,0,1) & \{3,8,11,15\} & 3 \\
5 & (0,0,0,0,1,0,1,1,0,1,0,1,1) & \{7,10,13\} & 4 \\
6 & (0,0,0,1,1,0,1,1,0,1,0,1) & \{5,8,12\} & 4 \\
7 & (0,1,0,1,1,0,1,0,1,1) & \{4,7,10\} & 4 \\
8 & (1,1,0,1,1,0,1,0,1) & \{2,5,9\} & 2 \\
9 & (1,1,0,1,0,1,1) & \{1,4,7\} & 1 \\
10 & (1,1,0,1,0,1) & \{2,6\} & 2 \\
11 & (1,0,1,1) & \{1,4\} & 1 \\
12 & (1,0,1) & \{3\} & 3 \\
13 & (1) & \{1\} & 1 \\
14 & \emptyf & & 4
\end{array}
$$\end{exmp}

As with the \akcode{k}, the construction of the \bkcode{k} ensures that $\bkburge(f)$ maps $\BFset$ into $\words{k+1}$, the set of words on $\{1,2,\ldots,k+1\}$ that end with a single copy of $k+1$.  Moreover, if  $\bkburge(f)=w_1\cdots w_r$ then we have $f=(\b_{w_1}\cdots \b_{w_r})(0)$, and we find that $f \mapsto \bkburge(f)$ has inverse $w \mapsto \b_w(0)$. (See the discussion following Example~\ref{exmp:burge}.)

\begin{lem}
\label{lem:bstats}
Suppose $f \in \Bset_i$. Then
\begin{enumerate}
\item	$\displaystyle
		\km{\dt f} = \begin{cases}
		\km{f}-1 &\text{if $\dt f \in \Bset_j$ with $i < j$ or $i=j\neq k+1$} \\
		\km{f} &\text{otherwise}.
		\end{cases}
		$
\item	$\displaystyle
		|\dt f| = 
		\begin{cases}
		|f|-k\km{f} &\text{if $i=k+1$} \\
		|f|-k\km{f} + (k -i) &\text{otherwise}. \\
		\end{cases}
		$
\end{enumerate}
\end{lem}

\begin{proof}
The proof is identical to that of Lemma~\ref{lem:stats},  except for relying on~\eqref{eq:irk_bstar} in place of~\eqref{eq:irk_astar}. In particular, with $g=\dt f$ we have $f=\b_i(g)$ and
$$
\Irk(f)=
\begin{cases}
    \Ilk(g)+k &\text{if $i=k+1$} \\
    \{i\}\cup \left(\Irk(\b_{k+1}(g_i, g_{i+1},\ldots)+(i-1)\right)
    & \text{otherwise}.
\end{cases}
$$
This gives
$$
\km f =
\begin{cases}
    \km{g} &\text{if $i=k+1$} \\
    1+\km{g_i,g_{i+1},\ldots} &\text{otherwise}.
\end{cases}
$$
and
\begin{equation*}
|f|=
\begin{cases}
|g| + k\km{g} &\text{if $i=k+1$} \\
|g| + i + k\km{g_{i},g_{i+1},\ldots} &\text{otherwise}.
\end{cases}
\end{equation*} 
We then analyse separately the cases $j<i$ and $j\ge i$ and find
\begin{align*}
j<i &\implies 
\Irk(g)=\{j\}\cup\left(\Irk(g_i,g_{i+1},\ldots)+(i-1)\right) \\
&\implies
\km{g}=1+\km{g_i,g_{i+1},\ldots} \\
\intertext{and}
i \leq j &\implies
\Irk(g)=\Irk(g_i,g_{i+1},\ldots)+(i-1) \\
&\implies 
\km{g}=\km{g_i,g_{i+1},\ldots}.
\end{align*}
Combining the above observations yields the desired conclusions.
\end{proof}

Finally we obtain the following analogue of Theorem~\ref{thm:Fburge} for partitions with distinct parts.  

\begin{thm}
\label{thm:Fbburge}
The encoding $f \mapsto \bkburge(f)$ is a bijection from $\BFset$ to $\words{k+1}$. Moreover, if $\w = \bkburge(f)$ and  if $a_i$ is the number of occurrences of $i$ in $\w$, then: 
\begin{enumerate}
\item	$a_i = \sum_{j \equiv i \pmod{k}} f_j$, for $1 \leq i \leq k$
\item	$\len{f} = a_1 + \cdots + a_{k}$
\item	$\km{f} = \basc(\w)$
\item	$|f| = k\bmaj(\w) - \sum_{i=1}^{k} (k-i)a_i$.
\end{enumerate}
\end{thm}

\begin{proof}
Repeat the proof of Theorem~\ref{thm:Fburge} but use Lemma~\ref{lem:bstats} in place of Lemma~\ref{lem:stats}.
\end{proof}

\begin{exmp}
In Example~\ref{exmp:bburge} we determined the \bkcode{3} of
$\l=(22,18,16,14,13,11,9,5,2)$ to be
$w=22434442121314$.
Thus we have $\bAsc \w=\{1,2,4,9,11,13\}$,  $\basc \w=6$, $\bmaj \w=40$, and $(a_1,a_2,a_3)=(3,4,2)$.  Observe that $\l$ has $\len{\l}=a_1+a_2+a_3=9$ parts, with $a_i$ of these congruent to $i \pmod{3}$. Moreover, $\mu_3(\l)=6$ and $|\l|=110=3\bmaj \w - 2a_1 - a_2$, in agreement with Theorem~\ref{thm:bburge}.
\end{exmp}

\begin{exmp}
The table below lists all partitions of 9 into distinct parts along with their \bkcode{4}s and associated data.  The reader can verify that $|\l| = 4\bmaj(\w) - 3a_1-2a_2-a_3$,  $\len{\l}=a_1+a_2+a_3+a_4$, and $\mu_k(\l)=\basc(w)$ in each case.
$$
\begin{array}{c|c|r|c|c|c|c|c|c}
\l & f & w=\bburge{4}(f) & \basc(\w) & \bmaj(\w) & a_1 & a_2 & a_3 & a_4 \\
\hline
(9) & (0,0,0,0,0,0,0,0,1) & 5515 & 1 & 3 & 1 & 0 & 0 & 0 \\
(8,1) & (1,0,0,0,0,0,0,1) & 145 & 2 & 3 & 1 & 0 & 0 & 1\\
(7,2) & (0,1,0,0,0,0,1) & 235 & 2 & 3 & 0 & 1 & 1 & 0\\
(6,3) & (0,0,1,0,0,1) & 5325 & 1 & 3 & 0 & 1 & 1 & 0\\
(6,2,1) & (1,1,0,0,0,1) & 2215 & 2 & 4 & 1 & 2 & 0 & 0\\
(5,4) & (0,0,0,1,1) &  5425 & 1 & 3 & 1 & 0 & 0 & 1\\
(5,3,1) & (1,0,1,0,1) &  1315 & 2 & 4 & 2 & 0 & 1 & 0\\
(4,3,2) & (0,1,1,1) & 4325 & 1 & 3 & 0 & 1 & 2 & 0
\end{array}
$$
\end{exmp}

\subsection{A Generalized Inversion Encoding}
\label{sec:invconst}


\newcommand{\g}{\delta}
\newcommand{\gpt}{\pt}
\newcommand{\gpti}[1]{\pt^{#1}}
\newcommand{\shift}{\Delta}
\newcommand{\ishift}{\Delta^{\!-}}
\newcommand{\Gset}{\mathcal{C}}
\newcommand{\gkcode}[1]{$[\delta,#1]$-code}
\newcommand{\gkinv}[1][k]{\Omega_{\g,#1}}
\newcommand{\circwords}[1]{\mathcal{W}_{#1}^\circ}


We now digress somewhat to introduce another encoding  of partitions as $k$-ary words that is quite different from $\akburge{}$ and $\bkburge{}$ but constructed in a similar manner.  This is a ``$(k-1)$-modular''  generalization of the well-known correspondence between partitions and lattice paths, whereby a path describes the boundary of a Young diagram.  (It reduces to this correspondence in case $k=2$.)   The correspondence is not directly related to $k$-measure and its construction does not rely on the cluster sets $\Ilk(f)$ and $\Irk(f)$ used previously.  

\begin{convention}
As before, $k \geq 2$ will be fixed throughout.
\end{convention}
  
For $1 \leq i < k$ we define the selective right shift operator
$\map{\shift_i}{\Fset}{\Fset}$ by
$$
    (\shift_i(f))_j 
    =
    \begin{cases}
        f_{j-k+1} &\text{if $j\equiv i\pmod{k-1}$} \\
        f_j &\text{otherwise}.
    \end{cases}
$$ 
Let $\ishift_i$ be the analogous left shift.  
That is, $\shift_i$ (resp. $\ishift_i$) acts on $f$ by shifting the subsequence $(f_i, f_{i+(k-1)}, f_{i+2(k-1)},\ldots)$ to the right (left) by $k-1$. Note that the $i$-th entry of $\shift_i(f)$ becomes 0 due to our convention that $f_j < 0$ for $j \leq 0$.  For example, with $k=4$ and $f=(1,2,3,4,5,6,0,\ldots)$ we have $\shift_2(f)=(1,0,3, 4,2,6,0,5,0,\ldots)$
and $\ishift_2(f)=(1,5,3,4,0,6,0,\ldots)$.

We now follow the template from the previous sections:
\begin{itemize}
\item
Define transformations $\g_1,\ldots,\g_k$ on $\Fset$ by
$$
    \g_j := 
    \begin{cases}
    \up{j} \shift_1 \shift_2 \cdots \shift_{j-1} & \text{if $1 \leq j < k$} \\
    \shift_1 \shift_2 \cdots \shift_{k-1} &\text{if $j=k$}.
    \end{cases}
$$
Note that $\g_k$ is the complete right $(k-1)$-shift on $\Fset$.

For example, with $k=4$ and $f=(1,0,3,2,0,1,0,2,4)$ we have:
\begin{align*}
	\g_1(f) &= (2,0,3,2,0,1,0,2,4) \\ 
	\g_2(f) &= (0,1,3,1,0,1,2,2,4,0) \\
	\g_3(f) &= (0,0,4,1,0,1,2,0,4,0,2) \\
	\g_4(f) &= (0,0,0,1,0,3,2,0,1,0,2,4).
\end{align*}

\item
Define subsets $\Gset_1, \ldots, \Gset_k$ of $\Fset$ by
$$
    \Gset_j := 
    \begin{cases}
    \{f  \,:\, f_1 = f_2 = \cdots = f_{j-1}=0, \, f_j > 0 \} &\text{if $1 \leq j < k$} \\
    \{f  \,:\, f_1 = f_2 = \cdots = f_{k-1}=0 \} &\text{if $j=k$}.
    \end{cases}
$$
Note that these sets partition $\Fset$, and  $\g_j$ is a bijection from $\Fset$ to $\Gset_j$ for all $j$.
  
\item	  Let $\map{\pt}{\Fset}{\Fset}$ denote the $k$-to-1 map  that restricts to $\g_j^{-1}$ on $\Gset_j$.  Explicitly, if $j = \min\{ i \,:\, f_i \neq 0\}$ then we have
$$
    \pt(f) := 
    \begin{cases}
    \ishift_1 \ishift_2 \cdots \ishift_{j-1}\down{j}(f) & \text{if $1 \leq j < k$} \\
    \ishift_1 \ishift_2 \cdots \ishift_{k-1}(f) &\text{otherwise}.
    \end{cases}
$$

For example, with $k=4$ and $f=(0,0,0,1,0,4,2,0,1,0,1)$ we have:
\begin{align*}
	\pt(f) &= (1,0,4,2,0,1,0,1) \\ 
	\pt^2(f) &= (0,0,4,2,0,1,0,1) \\
	\pt^3(f) &= (2,0,3,0,1,1) 
\end{align*}
\end{itemize}

As before we have $0 \leq |\gpt(f)| < |f|$ for $f \neq \emptyf$, which permits the following definition.

\begin{defn}
Let $f \in \Fset$ and let $r$ be the least nonnegative integer such that $\gpti{r}(f)=(0)$.  The \gkcode{k} of $f \in \Fset$ is the word $\gkinv(f) := w_0 \cdots w_r$ on $\{1,\ldots,k\}$ defined by $w_0=k$ and  $w_{j}=i \Leftrightarrow \gpti{j-1}(f) \in \Gset_i$ for $1 \leq j \leq r$.
\end{defn}

Note that $\gkinv(f)$ is not an element of $\words{k}$ except when $f=\emptyf$. Instead, the \gkcode{k} of $f$ is defined to be a word on $\{1,\ldots,k\}$ that {begins} with $k$ and does not end with $k$ unless $f=\empty$ (whose code is a single $k$).

\begin{exmp}
\label{exmp:gkcode}
Let $f = (1,2,0,0,3)$.  Computation of the $\g$-encodings 
$\gkinv[2](f)=21211222111$,
$\gkinv[3](f)=3122111$  and $\gkinv[4](f)=41224222$   are detailed below (left, center, and  right, respectively).
$$
\begin{array}[t]{ccc}
\hspace{5cm} & \hspace{5cm} & \hspace{5cm}\\
\begin{array}[t]{c|l|c}
i & \pt^{i-1}f & w_i \\ \hline
1 & (1,2,0,0,3) & 1 \\
2 & (0,2,0,0,3) & 2 \\
3 & (2,0,0,3) & 1 \\
4 & (1,0,0,3) & 1 \\
5 & (0,0,0,3) & 2 \\
6 & (0,0,3) & 2 \\
7 & (0,3) & 2 \\
8 & (3) & 1 \\
9 & (2) & 1 \\
10 & (1) & 1 \\
11 & \emptyf & 
\end{array} 
& 
\begin{array}[t]{c|l|c}
i & \pt^{i-1}f & w_i \\ \hline
1 & (1,2,0,0,3) & 1 \\
2 & (0,2,0,0,3) & 2 \\
3 & (0,1,3) & 2 \\
4 & (3) & 1 \\
5 & (2) & 1 \\
6 & (1) & 1 \\
7 & \emptyf & 
\end{array} 
& 
\begin{array}[t]{c|l|c}
i & \pt^{i-1}f & w_i \\ \hline
1 & (1,2,0,0,3) & 1 \\
2 & (0,2,0,0,3) & 2 \\
3 & (0,1,0,0,3) & 2 \\
4 & (0,0,0,0,3) & 4 \\
5 & (0,3) & 2 \\
6 & (0,2) & 2 \\
7 & (0,1) & 2 \\
8 & (0) & 
\end{array}
\end{array}
$$
Let $\l=(5,5,5,2,2,1)$ be the partition corresponding to $f$. Observe that  $\gkinv[2](f)$ describes the  boundary path of the Young diagram of $\l$, reading from left-to-right and interpreting symbols 2 and 1 as east and north steps, respectively.
See Figure~\ref{fig:boundarypath}.
\end{exmp}
\begin{figure}[t]
\centering
\includegraphics[height=2.5cm]{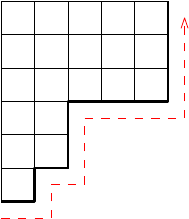}
\caption{The partition $\l=(5,5,5,2,2,1)$ is traced out by its \gkcode{2}  $21211222111$ by  interpreting $2$ and $1$ as east and north steps.}
\label{fig:boundarypath}
\end{figure}

The geometric interpretation of $\gkinv[2](f)$ noted in the above example is easily seen to hold for any $f$.  Indeed, when $k=2$ the operators $\g_1$ and $\g_2$ can be interpreted as acting on Young diagrams by appending an extra row of length 1 and duplicating the leftmost column, respectively.  This bears the familiar relation $|\l|=\inv(w)$.

The situation for general $k \geq 2$ is described by  the following \gkcode{k} analogue of Theorems~\ref{thm:Fburge} and~\ref{thm:Fbburge}.  The proof is straightforward and omitted. 

\begin{defn}
Let $\circwords{k}$ be the set of words on $\{1,\ldots,k\}$ of the form $w_0\cdots w_n$, where $w_0=k$ and $w_n \neq k$ (unless $n=0$).   
\end{defn}

\begin{thm}
\label{thm:inv}
The encoding $f \mapsto \gkinv(f)$ is a bijection from $\Fset$ to $\circwords{k}$, and if $a_i$ is the number of occurrences of $i$ in $w=\gkinv(f)$ then we have: 
\begin{enumerate}
\item	$a_i = \sum_{j \equiv i \pmod{k-1}} f_j$, for $1 \leq i < k$
\item	$\len{f}=a_1+\cdots+a_{k-1}$
\item	$|f| = (k-1)\inv(w) - \sum_{i=1}^{k-1}(k-1-i)a_i$
\end{enumerate}
\end{thm}


\section{Some Eulerian Calculus}
\label{sec:eulerian}

In this section we  review some classical $q$-enumeration and discuss a generalization of MacMahon's well-known generating series for the  pair $(\maj, \des)$ over rearrangement classes of words.  This will  be used to deduce Theorem~\ref{thm:main} as a corollary of our encodings of partitions.

\subsection{$q$-Binomial Coefficients}
  
Recall that the \emph{$q$-binomial coefficient}
$[\begin{smallmatrix} c \\ d \end{smallmatrix}]$ is  defined by
\begin{equation}
\qbinom{c}{d} := 
\begin{cases}
	\displaystyle \frac{\qf{q}{q}{c}}{\qf{q}{q}{d}\qf{q}{q}{c-d}}
	&\text{if $c \geq d \geq 0$} \\
	0 &\text{otherwise.}
\end{cases}
\label{eq:qbinomdefn}
\end{equation}
It is well known that $\tqbinom{c}{d}$ is a polynomial in $q$ in which the coefficient of $q^n$ counts partitions $\lambda$ of $n$ with at most $d$ parts, each at most $c-d$. (Equivalently, the Young diagram of $\l$ fits into an $d \times (c-d)$ rectangle.). From this combinatorial interpretation it is clear that
\begin{equation}
\label{eq:limit}
\lim_{a \rightarrow \infty} \qbinom{a+n}{a} = \sum_{\l \in \Pset} q^{|\l|} = \frac{1}{\qf{q}{q}{n}}.
\end{equation}
The  following familiar forms of the \emph{$q$-binomial theorem} are also seen to follow from  elementary decompositions of partitions; see, e.g.,~\cite{And}.
\begin{align}
\label{eq:binom1}
\qf{-qu}{q}{n} &= 
\sum_{a \geq 0} q^{\binom{a+1}{2}} \qbinom{n}{a} u^a \\
\label{eq:binom2}
\frac{1}{\qf{u}{q}{n+1}} &=
\sum_{a \geq 0} \qbinom{a+n}{n} u^a. 
\end{align}

\subsection{Enumeration on Rearrangement Classes}
\label{sec:enumeration}

\begin{convention}
Throughout this section (unless otherwise stated), $r$ is a positive integer and $X$ will denote a nonempty set $\{x_1,\ldots,x_r\}$ that is totally ordered by $x_1 < \cdots < x_r$. 
\end{convention}

\begin{defn}
For a sequence
  $\vec{a}=(a_1,\ldots,a_r)$ of nonnegative integers, let $\Wset(\vec{a})$ be the \emph{rearrangement class} of $X^*$ consisting of all words of length $|\vec{a}|:=a_1+\cdots+a_r$  that contain  $a_i$ copies of $x_i$, for $1 \leq i \leq r$. 
(Equivalently, $\Wset(\vec{a})$ consists of all permutations of the word $x_1^{a_1} x_2^{a_2} \cdots x_r^{a_r}$.) 
\end{defn}

 MacMahon~\cite{Mac-1} initiated the study of the major index and found that its distribution over $\Wset(\vec{a})$ was described by the polynomial
\begin{equation}
\label{eq:distribution}
	 \sum_{w \in \Wset(\vec{a})}  q^{\maj w} 
		=
	\frac{\qf{q}{q}{|\vec{a}|}}{\qf{q}{q}{a_1} \cdots \qf{q}{q}{a_r}}.
\end{equation}
Note that the RHS is the \emph{$q$-multinomial coefficient}
$[\begin{smallmatrix} a \\ a_1,\ldots,a_r \end{smallmatrix}]$, a natural extension of~\eqref{eq:qbinomdefn}.  He later demonstrated that $\inv(w)$ shares the same distribution, i.e.
\begin{equation}
\label{eq:invdistribution}
	 \sum_{w \in \Wset(\vec{a})}  q^{\inv w} 
		=
	\frac{\qf{q}{q}{|\vec{a}|}}{\qf{q}{q}{a_1} \cdots \qf{q}{q}{a_r}},
\end{equation}
A  combinatorial explanation for this equidistribution was provided by Foata~\cite{Foata},  who described a recursive bijection $\map{\foata}{\Wset(\vec{a})}{\Wset(\vec{a})}$ (known as his ``\emph{second fundamental transformation}'') satisfying $\inv(\phi(w))= \maj(w)$ for all $w$.  We will again refer to this transformation  in Section~\ref{sec:connections} but it does not play a central role and we will not describe it further.  The reader is directed to~\cite[Ch. 10]{Lot} for details.

The generating series~\eqref{eq:distribution} can be refined  to address the joint distribution of $(\maj \w, \des \w)$ over $\Wset(\vec{a})$. In~\cite[II.\S462]{Mac-3}, MacMahon shows
\begin{equation}
\label{eq:macmahon}
	\sum_{\w \in \Wset(\vec{a})} q^{\maj \w} z^{\des \w}
	=\qf{z}{q}{|\vec{a}|+1}
	\sum_{n \geq 0} \qbinom{a_1+n}{a_1}\cdots \qbinom{a_r+n}{a_r} z^n.
\end{equation}
Our proof of Theorem~\ref{thm:main} relies on an extension of this identity to words on a \emph{bipartitioned alphabet} with a generalized notion of descent. 

Several such extensions were  described by Clarke and Foata~\cite{ClFo-2,ClFo-3} and these were subsequently unified and further generalized by Foata and Zeilberger~\cite{FoZ}. For our purposes it will suffice to introduce the  notion of an \emph{$s$-descent}, defined as follows in~\cite{ClFo-2,ClFo-3}.   (In fact we will only need cases $s=0$ and $s=r-1$.)
\begin{defn}
\label{defn:xdescent}
Let $0 \leq s \leq r$.  We call $\{x_{r-s+1},\ldots,x_{r}\}$ the \emph{$s$-large} elements of $X$.  If  $w=w_1\ldots w_n$ is a word on $X$, then an \emph{$s$-descent} in $w$ is an index $i$ such that either 
\begin{enumerate}
\item $i < n$ and $w_{i} > w_{i+1}$, or 
\item $i < n$ and $w_{i}=w_{i+1}$ where $w_{i}$ is $s$-large, or 
\item $i = n$ and $w_i$ is $s$-large.
\end{enumerate}
Let $\Des_s \w$ be the set of $s$-descents in $\w$, let $\des_s \w := |\Des_s \w|$ denote their number, and let $\maj_s \w := \sum_{i \in \Des_s \w} i$ denote their sum. 
\end{defn}

For example, let $X=\{1,2,3,4,5\}$ with the natural ordering and let $\w=5441324$.  Then $\Des_0 \w = \Des_1 \w = \{1,3,5\}$ and $\maj_0 \w = \maj_1 \w =9$, whereas $\Des_2 \w = \{1,2,3,5,7\}$ and $\maj_2 \w = 18$.   Evidently, $0$-descents coincide with the ``ordinary'' descents defined previously.

\begin{thm}[{\cite[Section 4]{ClFo-3}}]
\label{thm:foata}
Fix $0 \leq s \leq r$, and for nonnegative $\vec{a}=(a_1,\ldots,a_r)$ let
$$
	A_{\vec{a}}(z,q) =  \sum_{\w \in \Wset(\vec{a})} q^{\maj_s \w} z^{\des_s \w}.
$$
Then
\begin{equation*}
\label{eq:foatathm}
	\sum_{\vec{a} \geq 0} 
	 \frac{A_{\vec{a}}(z,q)}{\qf{z}{q}{|\vec{a}|+1}} \vec{u}^{\vec{a}}
= \sum_{n \geq 0} 
\frac{\qf{-qu_{r-s+1}}{q}{n}\cdots\qf{-qu_r}{q}{n}}{\qf{qu_1}{q}{n+1} \cdots \qf{qu_{r-s}}{q}{n+1}} z^n,
\end{equation*}
where $\vec{u}^{\vec{a}}=u_1^{a_1} \cdots u_r^{a_r}$.
\end{thm}

Our interest in $s$-descents stems from the observation that the \aascent{s} and \bascent{s} of a word $w \in \words{k}$ are exactly the $0$-descents and $(k-1)$-descents of $w$ under the {reverse}  ordering of  $\{1,2,\ldots,k\}$.  We wish to track the distribution of these ascents over $\words{k}$,  so we seek a restriction of Theorem~\ref{thm:foata} that counts only words ending with a single copy of the \emph{smallest} symbol $x_1$.

Using~\eqref{eq:binom1} and~\eqref{eq:binom2} it is easy to see that Theorem~\ref{thm:foata} is equivalent to 
\begin{equation}
\label{eq:foata}
	A_{\vec{a}}(z,q)
	= \qf{z}{q}{|\vec{a}|+1}\sum_{n \geq 0} z^n \prod_{i=1}^{r-s} \qbinom{a_i+n}{a_i} 
	 \prod_{j=r-s+1}^{r} q^{\binom{a_j+1}{2}} \qbinom{n}{a_j}.
\end{equation}
This identity  is proved bijectively in~\cite{ClFo-3} by updating MacMahon's approach to~\eqref{eq:macmahon}. (Note that it specializes to~\eqref{eq:macmahon} at  $s=0$.)   
We shall now show how the desired restriction of Theorem~\ref{thm:foata} is obtained from the limiting case $a_1 \rightarrow \infty$. 
Essentially, as $a_1$ grows, the words in $\Wset(\vec{a})$  with bounded major index are those that factor into a ``short'' word  ending in a single $x_1$  followed by a weakly increasing word in the small symbols $x_1,\ldots,x_{r-s}$. Thus $\lim_{a_1 \rightarrow \infty} A_{\vec{a}}(z,q)$ provides a predictable overcount of the words ending in a single $x_1$.

\begin{cor}
\label{cor:foata}
For a sequence $\vec{c}=(c_2,\ldots,c_{r})$ of nonnegative integers, let $\Wstar(\vec{c})$ be the set of  words on $X$ that contain $c_i$ copies of $x_i$ for $2 \leq i \leq r$ and end with a single $x_1$.  Then for $0 \leq s < r$ we have
$$
 \sum_{\vec{c} \geq 0} 
  \sum_{\w \in \Wstar(\vec{c})} q^{\maj_s \w} z^{\des_s \w} \vec{u}^\vec{c}
	= \qf{z}{q}{\infty}\sum_{n \geq 0} 
		\frac{\qf{-qu_{r-s+1}}{q}{n}\cdots\qf{-qu_r}{q}{n}}{\qf{q}{q}{n} \qf{qu_2}{q}{n} \cdots \qf{qu_{r-s}}{q}{n}} z^n,
$$
where $\vec{u}=(u_2,\ldots,u_{r})$ are indeterminates and  $\vec{u}^\vec{c} = u_2^{c_2} \cdots u_{r}^{c_{r}}$.
\end{cor}


\begin{proof}
Let $\vec{a}=(a_1,\ldots,a_{r}) \geq 0$  and 
suppose $\w =\w_1 \cdots \w_{|\vec{a}|} \in \Wset(\vec{a})$  satisfies $\maj_s \w < a_1$.   We claim that $\w$ is of the form 
\begin{equation}
\label{eq:wfactor}
	\w = \wstar\, x_1^{a_1-c_1} x_2^{a_2-c_2} \cdots x_{r-s}^{a_{r-s}-c_{r-s}},
\end{equation}
where $\wstar$ ends with a single $x_1$,  contains $c_i$ copies of $x_i$ for $1 \leq i \leq r-s$, and contains $a_i$ copies of $x_i$ for $r-s < i \leq r$.

To see this, let $t$ be the largest (rightmost) $s$-descent of $\w$, taking $t=0$ if this does not exist.   Since $t \leq \maj_s \w < a_1$ we have $\w_i=x_1$ for some $i > t$.  But the definition of $t$ gives $\w_{t+1} \leq \w_{t+2} \leq \cdots \leq \w_{|\vec{a}|}$, where $\w_{|\vec{a}|}$ is not $s$-large, and so it must be the case that $\w_{t+1}=x_1$  and $\w_i \in \{x_1,\ldots,x_{r-s}\}$ for all $i > t$. Moreover, if $t>0$ then we have $w_t \neq x_1$ because $t$ is an $s$-descent and $x_1$ is not $s$-large (i.e. $s < r$).   Thus $w$ is  of the claimed form, with $\wstar=w_1 w_2 \cdots w_{t+1}$. 

We have shown that any word $w \in \Wset(\vec{a})$ with $\maj_s w < a_1$ factors as in~\eqref{eq:wfactor}.  Moreover, observe that $(\maj_s \wstar,\des_s \wstar)=(\maj_s \w, \des_s \w)$.
Accordingly, if $m < a_1$ then the coefficient of $q^m z^d$ in $A_{\vec{a}}(z,q)$  counts tuples $(c_1,\ldots,c_r,\wstar)$ such that
\begin{enumerate}[(i)]
\item	$1 \leq c_1 \leq a_1$,
\item	$0 \leq c_i \leq a_i$ for $2 \leq i \leq r-s$ and  $c_i=a_i$ for $r-s < i \leq r$,
\item	$\wstar \in \Wset(c_1,\ldots,c_r)$ ends in a single $x_1$, and
\item	$(\maj_s \wstar,\des_s \wstar)=(m,d)$. 
\end{enumerate}
In the limit $a_1 \rightarrow \infty$,  identity~\eqref{eq:foata} therefore yields
\begin{equation}
\label{eq:afterlimit}
  \sum_{\vec{c}}\,\, \sum_{\wstar \in \Wstar(\vec{c}) } q^{\maj_s \wstar} z^{\des_s \wstar} 
=
  \qf{z}{q}{\infty}\sum_{n \geq 0} \frac{z^n}{\qf{q}{q}{n}} \prod_{i=2}^{r-s} \qbinom{a_i+n}{a_i} 
	 \prod_{j=r-s+1}^{r} q^{\binom{a_j+1}{2}} \qbinom{n}{a_j},
\end{equation}
where the outer sum on the left extends over all $\vec{c}=(c_2,\ldots,c_{r})$ satisfying (ii). Note that~\eqref{eq:limit} has been applied on the right.

Now multiply~\eqref{eq:afterlimit} by $u_2^{a_2} \cdots u_r^{a_r}$  and sum over all $a_2,\ldots,a_r \geq 0$.
On the LHS this leads to
\begin{equation*}
	\sum_{\vec{c} \geq 0}\,\, \sum_{\wstar \in \Wstar(\vec{c})} q^{\maj_s \wstar} z^{\des_s \wstar} \vec{u}^{\vec{c}} \cdot \prod_{i=2}^{r-s}  (1-u_i)^{-1}
\end{equation*}
upon switching the order of summation, while
\eqref{eq:binom1} and~\eqref{eq:binom2} can be applied on the RHS to give
$$
	 \qf{z}{q}{\infty}\sum_{n \geq 0} 
		\frac{\qf{-qu_{r-s+1}}{q}{n}\cdots\qf{-qu_r}{q}{n}}{\qf{q}{q}{n} \qf{u_2}{q}{n+1} \cdots \qf{u_{r-s}}{q}{n+1}} z^n.
$$
This completes the proof, since $\qf{u_i}{q}{n+1}=(1-u_i)\qf{qu_i}{q}{n}$.
\end{proof}

\subsection{Proof of Theorem~\ref{thm:main}}

Theorem~\ref{thm:main} follows quickly from piecing together Theorems~\ref{thm:burge} and~\ref{thm:bburge}  with Corollary~\ref{cor:foata}.

As noted above, if $\w \in \words{k}$ then the \aascent{s} of $w$  are the same as its $0$-descents with respect to the reverse ordering $k < k-1 < \cdots < 1$ of  $\{1,\ldots,k\}$.    According to Theorem~\ref{thm:burge}, the LHS of~\eqref{eq:main1} is therefore obtained by applying Corollary~\ref{cor:foata} to $X=\{k,k-1,\ldots,1\}$ with $s=0$ and then making the substitutions $q \leftarrow q^{k-1}$ followed by $u_{k+1-i} \leftarrow y_i/q^{k-1-i}$ for $1 \leq i \leq k-1$. It is easy to check that doing so yields the RHS of~\eqref{eq:main1}.

Similarly, the \bascent{s}  of a word $w \in \words{k+1}$ match its $k$-descents with respect to the reverse ordering of $\{1,\ldots,k+1\}$.  According to Theorem~\ref{thm:bburge} we obtain the LHS of~\eqref{eq:main2} by applying  Corollary~\ref{cor:foata}  to $X=\{k+1,\ldots,1\}$ with $s=k$ and substituting $q \leftarrow q^{k}$ followed by $u_{k+2-i} \leftarrow y_{i}/q^{k-i}$, $1 \leq i \leq k$.  This results in the RHS of~\eqref{eq:main2}.

\section{Interaction with Foata's  Transformation}
\label{sec:connections}

\newcommand{\hk}{\theta}
\newcommand{\hkburge}[1][k]{\Omega_{\theta,#1}}
\newcommand{\myfoata}{F}
\newcommand{\foatak}[1][k]{\phi}
\newcommand{\goback}[1]

 Consider the map  $\myfoata$ on $\{1,\ldots,k\}^*$ defined by $\myfoata(w) := (\foata(w))^{-1}$, where $v^{-1}$ is the reverse reading of the word $v$ and $\foatak$ denotes Foata's second fundamental transformation  invoked with respect to the reverse ordering of $\{1,\ldots,k\}$.  
 Then  $\myfoata(w)$ is a rearrangement of $w$ which satisfies $\inv(\myfoata(w))=\amaj(w)$.  Moreover, from the construction of $\foatak$ (see~\cite[Ch. 10]{Lot}) it is easy to verify that a word $w \in \{1,\ldots,k\}^*$ of length $n > 1$ satisfies  $\{w_{n-1} \neq k, w_n=k\}$ if and only if $v=\myfoata(w)$ satisfies $\{v_1 \neq k, v_n=k\}$. Hence $\myfoata$ is a bijection from $\words{k}$ to $\circwords{k}$.

 Theorems~\ref{thm:burge} and~\ref{thm:inv} lead us to consider the following diagram. 
\begin{equation}
\label{eq:diagram}
\begin{tikzcd}[row sep=2.5em]
\Pset \arrow[rr,dashed] \arrow[d,"\akburge"] &&  \Pset  \\
\words{k} \arrow[rr,"\myfoata"] && \circwords{k} \arrow[u,"\gkinv^{-1}"]
\end{tikzcd}
\end{equation}
The upper composition $\gkinv \circ \myfoata\circ \akburge^{-1}$ is a nontrivial size-preserving bijection on $\Pset$, and so it is natural to ask whether it has any interesting properties.  For $k=2$ the answer is affirmative, as we now explain.

Foata's transformation has a simple non-recursive description when the underlying alphabet is of size 2. This has been explored in detail by Sagan and Savage~\cite{Sag-Sav}. Let
$$
	w = 2^{m_0} 1^{n_0} 2^{m_1} 1^{m_1} \cdots 2^{m_r} 1^{n_r} 2 
$$
be a general word in $\words{2}$, 
where  $m_0 \geq 0$ and all other $m_i, n_i$ are positive.
Then from~\cite[Prop. 2.2]{Sag-Sav} we obtain
\begin{align}
\label{eq:sagan}
	\myfoata(w)
	&=
	2 1^{n_r-1} 2 1^{m_{r-1}-1} \cdots 2 1^{n_0-1} 
	2^{m_0} 1 2^{m_1-1} 1 2^{m_2-1} \cdots 1 2^{m_r-1} 1
\end{align}
This word describes the northeast boundary of a Young diagram, with  2's and 1's indicating east and north steps, respectively. The corresponding partition is $\l = \gkinv[2]^{-1}(\myfoata(w))$, and it is easy to verify from~\eqref{eq:sagan} that $\l$ is of size $\inv(\myfoata(w))=\maj(w)$ and  has Durfee square of side $\aasc(w)=d$.  In fact, $\aAsc(w)$ is  the set of \emph{hook lengths} of $\l$, defined below.   (See~\cite[Cor. 2.4]{Sag-Sav} and also~\cite[Lem. 3.2]{CorElS}.)

\begin{defn}
Let $\l$ be a partition with Durfee square of side $d$.  
For $1 \leq i \leq d$, the $i$-th \emph{diagonal hook} of  $\l$ is the set $H_i = \{(i+x,i-y) \in \l \,:\, x,y\geq 0\}$ of all cells in its Young diagram directly below or directly to the right of $(i,i)$. 
The quantities $|H_i|$ are called the \emph{hook lengths} of $\l$.
\end{defn}

\begin{exmp}
Let  $w=22121112112$, with $\aAsc(w) = \{3,7,10\}$ and $\amaj(w)=20$. Applying~\eqref{eq:sagan} yields  $\myfoata(w)=21211222111$, which describes the boundary of the partition $\l=(5,5,5,2,2,1)$ displayed in Figure~\ref{fig:boundarypath}. Thus $\myfoata(w)$ is the \gkcode{2} of $\l$.  Note that $|\l|=20$, its Durfee square is of size $\aasc(w)=3$, and its hook lengths are 10, 7, and 3. 
\end{exmp}

From these observations and Theorem~\ref{thm:burge} we conclude that $\gkinv[2]^{-1} \circ \myfoata\circ \akburge[2]$ is a size- and length-preserving bijection on $\Pset$ that translates 2-measure into size of Durfee square, i.e., number of hooks.  This gives a combinatorial (though recursively bijective) proof of Theorem~\ref{thm:abd}, as sought by the authors of~\cite{ABD}. We now demonstrate that this same bijection had already been stated in a different guise by Burge.

In his seminal papers~\cite{Bur-1,Bur-2} Burge gave three distinct correspondences between partitions and binary words.  One of these, $\akburge[2]$, is built from  transformations $\{\a_1,\a_2\}$ on $\Pset$ as described previously.  Another is constructed similarly, but using two different transformations (which we call $\hk_1$ and $\hk_2$) that act on  partitions  by enlarging  diagonal hooks along the vertical and horizontal axes separately.
The precise definitions are best given in terms of \emph{Frobenius notation}.

\newcommand{\smallfrob}[2]{\left(\begin{smallmatrix} #1 \\ #2 \end{smallmatrix}\right)}
\newcommand{\frob}[2]{\Big(\begin{matrix} #1 \\ #2 \end{matrix}\Big)}

Let $H_1,\ldots,H_d$ be the diagonal hooks of a partition $\l$, and let $a_i$ (respectively, $b_i$) be the number of cells in $H_i$ strictly to the right (resp.,  below) cell $(i,i)$.  Clearly $a_1 > a_2 > \cdots > a_d \geq 0$ and $b_1 > b_2 > \cdots > b_d \geq 0$, with $|H_i| = a_i+b_i+1$.  Moreover, since the $a_i$ and $b_i$ completely specify $\l$, we can alternatively denote $\l$ by
$\smallfrob{a_1 & \cdots & a_d}{b_1 & \cdots &b_d}.$ 
For example, $\l=(5,5,5,2,2,1)$ is denoted $\smallfrob{4 & 3& 2}{5&3&0}$. 

The transformations $\hk_1$ and $\hk_2$ are then defined on $\Pset$ as follows:
\begin{align*}
	\hk_1\!\smallfrob{a_1 & \cdots & a_d}{b_1 & \cdots & b_d}
	&:= \begin{cases}
	    \smallfrob{a_1 & \cdots & a_d & 0}{b_1+1 & \cdots & b_d+1 & 0} &\text{if $a_d > 0$,} \\
	     \smallfrob{a_1 & \cdots & a_d}{b_1+1 & \cdots & b_d+1} &\text{otherwise}.
	\end{cases} \\
	\hk_2\!\smallfrob{a_1 & \cdots & a_d}{b_1 & \cdots & b_d}
	&:= 
	    \smallfrob{a_1+1 & \cdots & a_d+1}{b_1 & \cdots & b_d}. \end{align*}
Note that $\hk_1$ extends each hook vertically and also adds a new hook (of length 1) whenever it is sensible to do so.
See Figure~\ref{fig:hookbuilding}.  
\begin{figure}
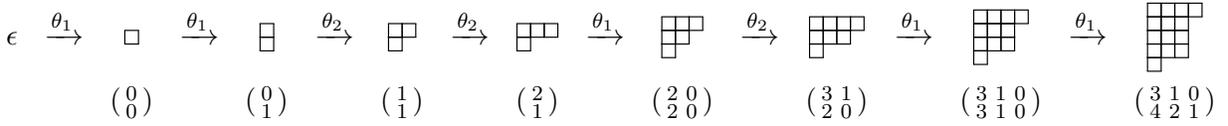

\newcommand{\andthen}[1]{\xrightarrow{#1}}
$$
\ytableausetup{boxsize=1ex,centertableaux}
\begin{array}{ccccccccccccccccc}
\epsilon & \andthen{\hk_1} &
\ydiagram{1} & \andthen{\hk_1} &
\ydiagram{1,1} & \andthen{\hk_2} &
\ydiagram{2,1} & \andthen{\hk_2} & \ydiagram{3,1} 
&\andthen{\hk_1} & \ydiagram{3,2,1} &\andthen{\hk_2}  & \ydiagram{4,3,1} &\andthen{\hk_1} & \ydiagram{4,3,3,1}
& \andthen{\hk_1} 
& \ydiagram{4,3,3,3,1}\\
\rule{0pt}{3ex}  & & \smallfrob{0}{0} & & \smallfrob{0}{1}
& & \smallfrob{1}{1} & & \smallfrob{2}{1} & & \smallfrob{2 & 0}{2 & 0} & & \smallfrob{3 & 1}{2 & 0} & &\smallfrob{3 & 1 & 0}{3 & 1 & 0}
& & \smallfrob{3 & 1 & 0}{4 & 2 & 1}
\end{array}
$$
\caption{Building $\l=(4,3,3,3,1)$ by iterative application of the operators $\hk_1$ and $\hk_2$.}
\label{fig:hookbuilding}
\end{figure}

The functions $\hk_1$ and $\hk_2$ are clearly injective and their images partition $\Pset$.  They therefore induce an encoding $\map{\hkburge[2]}{\Pset}{\words{2}}$ in the same manner as $\akburge[2]$ and $\gkinv[2]$ were constructed from $\{\a_1,\a_2\}$ and $\{\g_1,\g_2\}$ in  Sections~\ref{sec:burgeconst} and~\ref{sec:invconst}.    
For example, we get $\hkburge[2](4,3,3,3,1)=1212212$ by reading the chain of operations displayed in Figure~\ref{fig:hookbuilding} from  right-to-left. Burge presents this bijection in~~\cite[\S2]{Bur-1} and  notes that the partition $\l$ corresponding to $w$ is of size $\amaj(w)$, contains $\aasc(w)$ hooks, and is of length equal to the number of 1s in $w$. Indeed, it is routine to check that $\hkburge[2]^{-1}$ is precisely the map $\gkinv[2]^{-1} \circ \myfoata$ discussed above, i.e., the northeast diagonal of~\eqref{eq:diagram}.    (Burge goes on to derive a number of  interesting identities by carefully analyzing the correspondence 
 $\akburge[2]^{-1} \circ \hkburge[2]$ on $\Pset$.)


When $k > 2$ the map $\gkinv^{-1} \circ \myfoata$ continues to translate words $w \in \words{k}$ with $\amaj(w)=n$ into partitions $\l$ of size $n$.  However, the boundary-path interpretation of $w$ breaks down and  the meaning of $\aasc(w)$ in relation to $\l$ is unclear.   We are left to wonder if Burge's hook-building transformations $\{\hk_1,\hk_2\}$ admit a ``$(k-1)$-modular'' extension similar to how $\{\a_1,\a_2\}$ have been generalized to $\{\a_1,\ldots,\a_k\}$. 


\section{Schur's Theorem and Generalizations}
\label{sec:schur}

\newcommand{\incwords}[1]{\mathcal{W}^{\scriptscriptstyle +}_{#1}}
\newcommand{\gapset}[1]{G^{\pi}_{#1}}
\newcommand{\Dsep}[1][k]{\mathcal{D}_{#1}}
\newcommand{\Dstrict}[2][k]{\mathcal{D}_{#1,#2}}
\newcommand{\Dseppi}[1][k]{\mathcal{D}^{\pi}_{#1}}
\newcommand{\Dstrictpi}[2][k]{\mathcal{D}^{\pi}_{#1,#2}}

In this final section we show how the $[\b,k]$-encoding can be applied to give concise proofs of some well-known results concerning  gaps between successive parts of  partitions.

\begin{defn}
 Let $\Dsep$ be the set of partitions  $\l=(\l_1,\l_2,\ldots)$ such that $\l_i - \l_{i+1} \geq k$ for all $i < \len{\l}$. For $J \subset \{1,\ldots,k\}$,  let  $\Dstrict{J}$ be the set of all $\l \in \Dsep$ that additionally satisfy $\l_i - \l_{i+1} \neq k$ whenever $\l_i \equiv j \pmod{k}$ and $j \in J$. In particular,
 $\Dstrict{\emptyset}=\Dsep$.
\end{defn}

For instance: $\Dsep[1]=\Dset$,  $\Dsep[2]$ is the set of \emph{super distinct} partitions that arise in the celebrated Rogers-Ramanujan identities (see also~\cite{IKM} for an interesting algebraic manifestation), and $\Dstrict[4]{\{1,3\}}$ is the set of partitions in which parts differ pairwise by $\geq 4$ and odd parts differ pairwise by $\geq 6$.

\begin{defn}
Let $\incwords{k+1}$ be the subset of $\words{k+1}$ consisting of words  that do not contain any occurrences of a substring $ji$ with $i < j \leq k$.   
Equivalently, $\incwords{k+1}$ consists of all concatenations of  words of the form $1^{a_1} 2^{a_2} \cdots k^{a_k} (k+1)^m$ with $a_i \geq 0$ and $m > 0$.
\end{defn}

Let $\l \in \Dset$ and let $f \in \BFset$ be its frequency list.  Then $\l \in \Dsep$ if and only if there are no pairs $(s,t)$ with $f_s=f_t=1$ and  $0 < |t-s| < k$.  Let us say such a pair is \emph{tight} for $f$.  Now consider the iterative creation of $f$ according to its code $w=\bkburge(f)$, i.e. 
$$
	(0) \rightarrow \b_{w_n}(0) \rightarrow \b_{w_{n-1}}\b_{w_n}(0) \rightarrow \cdots \rightarrow w_{\b_1}\cdots w_{\b_n}(0) = f.
$$  
From definition of $\b_1,\ldots,\bstar$ it is easy to see that tight pairs are formed only by the application of $\b_i$ immediately followed by $\b_j$, where $i < j \leq k$.  Moreover, if $(s,t)$ is tight for $g$, then $\b_j(g)$ has either the tight pair $(t,s+k)$ or one of the form $(s,t')$ with $s < t' \leq t$. So once one tight pair is created, at least one will always remain. Similarly, pairs $f_s=f_{s+k}=1$ are formed only by successive applications of $\b_j$, where $j \equiv s \pmod{k}$, and again these are persistent.  Thus we have:

\begin{lem}
\label{lem:gap}
Let $\l=(\l_1,\l_2,\ldots) \in \Dset$ and let $w = \bkburge(\l)$.  Then $\l \in \Dstrict{J}$ if and only if $w \in \incwords{k+1}$ and $w$  avoids the substring $jj$ for all $j \in J$.
\end{lem}

Lemma~\ref{lem:gap} reflects the fact that $\bkburge$ acts rather trivially on $k$-distinct partitions.  Indeed, $\Dsep$ can also be characterized as the set of partitions  whose \emph{$k$-modular diagrams} contain  no two rows of the same length, and the
$[\b,k]$-code of any such $\l$ can  be read directly from its diagram.  See Figure~\ref{fig:kmodular}.
\begin{figure}
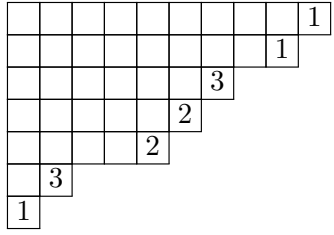

\begin{center}
\ytableausetup{boxsize=2.5ex}
\ytableaushort {
\ \ \ \ \ \ \ \ \ 1,
\ \ \ \ \ \ \ \  1,
\ \ \ \ \ \ 3,
\ \ \ \ \ 2,
\ \ \ \ 2,
\ 3,
1
}
\end{center}
\caption{The $3$-modular diagram of $\l=(28,25,21,17,14,6,1) \in \Dstrict[3]{\{3\}}$. The cell at the end of each row is filled with a weight $\in \{1,2,3\}$. All other cells are empty and of weight 3. The  code  $\bkburge[3](\l)=1344223411$ is  obtained by scanning from left-to-right and recording  the entry in bottom cell of each column, or $4$ ($=k+1$) if this cell is empty.  }
\label{fig:kmodular}
\end{figure}


There are a number of results in partition theory concerning restrictions on the gaps $\l_i-\l_{i+1}$ between consecutive parts.  One of the earliest and best-known of these the following theorem of Schur, which was later generalized and proved combinatorially by Bressoud~\cite{Bres}.

\begin{thm}[{Schur~\cite{Schur}}]
\label{thm:schur}
Let $s_{n}$ be the number of partitions of $n$ into distinct parts, each of which is $\equiv \pm 1 \pmod{3}$. Let $t_{n}$ be the number of partitions of $n$ in which all parts differ pairwise by at least 3, with parts divisible by 3 differing by at least 6.  Then $s_n=t_n$ for all $n \geq 0$. 
\end{thm}


We now present a tidy proof of Schur's Theorem via the $[\b,3]$-encoding.  
The following notation will be helpful.

\begin{defn}
For $w \in \words{k+1}$ let $|w|_\b := k\bmaj w - \sum_{i=1}^k(k-1)b_i$, where $b_i$ is the number of occurrences of $i$ in $w$.
\end{defn}

\begin{proof}[Proof of Theorem~\ref{thm:schur}]
Let $\mathcal{S}$ be the set of words on $\{1,2,4\}$ that end with a single 4 and let $\mathcal{T}$ be the set of words in $\incwords{4}$ that avoid substring 33.  Then taking $k=3$ in Theorem~\ref{thm:bburge} and Lemma~\ref{lem:gap} we have
\begin{equation}
\label{eq:sntn}
	s_n = |\{ w \in \mathcal{S} \,:\, \bmaj w=n \}|
	\qquad\text{and}\qquad	t_n =  |\{w \in \mathcal{T} \,:\, \bmaj w=n\}|.
\end{equation}
For $w \in \mathcal{T}$, let $S(w)$ be the word obtained from $w$ by replacing each occurrence of 34 with 21, with the caveat that a terminal occurrence of 34 is replaced with 214.   For instance, $S(4122\mathbf{34}1\mathbf{34}144\mathbf{34}) = 4122\mathbf{21}1\mathbf{21}414\mathbf{214}$.  We claim that $S$ is a bijection from $\mathcal{T}$ to $\mathcal{S}$ satisfying $|w|_\b =|S(w)|_{\b}$. The result then follows from~\eqref{eq:sntn}.

To prove the claim, first note that words $w \in \mathcal{T}$ can be uniquely decomposed into  blocks of the form $1^a 2^b 34 4^c$ or $1^a 2^b4 4^c$, with $a,b,c \geq 0$.  Thus $S^{-1}$ simply replaces every occurrence of $34$ with $21$.  
Now suppose $w=w_1 \cdots w_n \in \mathcal{T}$ has $w_i w_{i+1}=34$. Let $w'$ be the word obtained by changing this pair to $21$. 
For $j \neq i,i+1$ we  observe that $j \in \bAsc w \Leftrightarrow j \in \bAsc w'$.  Evidently $\bAsc \w$ contains $i$ but not $i+1$, while $\bAsc w'$ contains $i+1$ but not $i$. Thus $\bmaj w' = \bmaj w +1$. Since the numbers of 1's and 2's each increase by 1 in the transition from $w$ to $w'$, we get $|w'|_\b = |w|_\b + 3 - 2(1)-1(1) = |w|_\b$.
\end{proof}

Our proof of Schur's Theorem differs from Bressoud's~\cite{Bres},   but after some unwinding it can be seen to match  the iterative diagrammatic approach described by Pak in~\cite[\S4.5.1]{Pak}. The simplicity of our bijection $S$ is of course due to  details being hidden by the $[\b,3]$-encoding.

Note that $S$  preserves the number of \bascent{s}, i.e. $\basc S(w)=\basc w$, and thus the statement of Schur's Theorem  can  be refined in the obvious way to account for $3$-measure.   Moreover, with  no additional effort we can obtain the following generalization  due to Alladi and Gordon~\cite{AllGor-1}.  Schur's Theorem is the special case $M=3$, $a=1$, $b=2$. The case $a=1$, $b=M-1$ for arbitrary $M$ is due to Bressoud~\cite{Bres}.  (The result is stated in~\cite{AllGor-1,AllGor-2} without reference to $k$-measure.)

\newcommand{\Skr}[1][M]{\mathcal{S}}
\newcommand{\Tkr}[1][M]{\mathcal{T}}

\begin{thm}[{Generalized Schur's Theorem~\cite{AllGor-1,AllGor-2}}] 
\label{thm:genschur} Let $M \geq 3$ and fix positive integers $a,b$ with $0 < a < b < M$.  Let $\Skr$ be the set of partitions into distinct parts congruent to $a$ or $b \pmod{M}$.  Let $\Tkr$ be the set of partitions into distinct parts congruent to $a, b$ or $a+b \pmod{M}$ such that
\begin{itemize}
\item	all parts are $\geq a$, 
\item	the difference between any two parts is $\geq M$, and 
\item	the difference between any two parts $\equiv a+b \pmod{M}$ is  $> M$.  
\end{itemize}
For any $n, m \geq 1$, the number of partitions in $\Skr$ of size $n$ and $k$-measure $m$ equals the number of partitions in $\Tkr$ of size $n$ and $k$-measure $m$.\end{thm}

\begin{proof}
If $(a+b) \leq M$ then the proof mimics that of Theorem~\ref{thm:schur}, but taking $k =M$  and replacing symbols $\{1,2,3,4\}$ with $\{a,b,a+b,M+1\}$.  The exchange $(a+b)(M+1) \rightarrow ba$ on $w \in \mathcal{T}$ increases $\bmaj w$ by 1,  increases the numbers of $a$'s and $b$'s by one each, and decreases by one the number of $(a+b)$'s.  So the net change to $|w|_{\b}$ is again $0=M - (M-a)-(M-b)+(M-(a+b))$.

If $a+b > M$  we let $c=a+b-M$, noting that $0 < c < a < b$. Then $\Skr$ is  again in correspondence with words on $\{a,b,M+1\}$ that end with a single $M+1$. However $\Tkr$ is is now correspondence with words on $\{c,a,b,M+1\}$ that end in a single $M+1$, do not begin with $c$, and avoid the substrings $\{ca, cb, cc, ba\}$. The exclusion of first symbol $c$ ensures that all parts of the resulting partition are $\geq a$.  In this case we move from $\mathcal{S}$ to $\mathcal{T}$ by swapping each adjacent pair $(M+1)c$ for $ba$. We find that $\bAsc w$ is unaltered by this swap, so the net change to $|w|_{\b}$ is  $0=- (M-a)-(M-b)+(M-c)$.
\end{proof}

We emphasize that our proof of Theorem~\ref{thm:genschur} is not novel, but rather a reorganization of one that has come before. The $[\b,k]$-encoding of $\Dset$ into $\words{k+1}$ is very closely related to Alladi and Gordon's \emph{method of weighted words}, and ultimately the combinatorics underlying our approaches   is equivalent.  

For the sake of completeness, we close this section by briefly discussing a related class of partition identities analyzed in~\cite{AllGor-2}.

\newcommand{\pil}[1][\lambda]{{#1}^{\pi}}
\newcommand{\piw}[1][w]{{#1}^{\pi}}
\newcommand{\stabword}{\mathcal{V}_{k+1}^{j}}

Let $\Sym_k$ be the group of permutations on $\{1,\ldots,k\}$. We define an action $\l \mapsto \pil$ of $\Sym_k$ on $\Dset$ as follows.

\begin{defn}
If $\l \in \Dset$ and $\pi \in \Sym_k$, let   $\pil \in \Dset$  be is the  partition satisfying  
$
	qk+\pi(r) \in \pil \iff qk+r \in \l, 
$
for $q \geq 0$ and $1 \leq r \leq k$.
\end{defn}

Equivalently, $\pil$ is obtained from $\l$ by applying $\pi$ to the entries in its $k$-modular diagram.  For example,  $\pi=312 \in \Sym_3$ and $\l=(28,25,21,17,14,6,1)$ give $\pil = (30, 27, 20, 16, 13, 5, 3)$.

Some elementary bookkeeping yields the following translation of Lemma~\ref{lem:gap}.

\begin{lem}
Fix $k \geq 2$ and $\pi \in \Sym_k$. For $1 \leq r \leq k$ define 
$$
	\gapset{r} := \{\pi(r)+k-\pi(j) \,:\, 1 \leq j \leq r\} \cup [\pi(r)+k, \infty).
$$
For $J \subset \{1,\ldots,k\}$, let $\Dstrictpi{J}$ be the set of partitions $\l=(\l_1,\l_2,\ldots)$ such that 
\begin{equation*}
		\l_i \equiv \pi(r) \pmod{k} \implies
		\l_{i}-\l_{i+1} \in \gapset{r}, \qquad 1 \leq i < \len{\l},
\end{equation*}	
and $\l_i - \l_{i+1} \neq k$ whenever $\l_i \equiv \pi(j) \pmod{k}$ and $j \in J$.  Then  $\nu \in \Dstrict{J}$ if and only if $\pil[\nu] \in \Dstrictpi{J}$.
\end{lem}

\begin{exmp}
Let $k=3$ and $\pi=132$.  Then $\gapset{1} = \{3\} \cup [4,\infty)$, $\gapset{2} = \{5,3\} \cup [6,\infty)$, and $\gapset{3} = \{4,2,3\} \cup [5,\infty)$.  Thus $\Dstrictpi[3]{\{2\}}$ consists of partitions $\l$ satisfying
\begin{equation}
\label{eq:andrewsgap}
	\l_i -\l_{i+1} \geq 
	\begin{cases}
	3 &\text{if $\l_i \equiv \pi(1) \equiv 1 \pmod{3}$} \\
	5 &\text{if $\l_i \equiv \pi(2) \equiv 3 \pmod{3}$} \\
	2 &\text{if $\l_i \equiv \pi(3) \equiv 2 \pmod{3}$}.
	\end{cases}
\end{equation}
\end{exmp}

The transformation $\l \mapsto \pil$ is bijective from $\Dstrict{J}$ to $\Dstrictpi{J}$ but clearly not size-preserving.  This can be addressed via composition with an action of $\Sym_k$ on $\incwords{k+1}$.
   
\begin{defn}
For $\pi \in \Sym_k$ and $w\in \incwords{k+1}$, define the word $\piw \in \incwords{k+1}$ as follows.  First  decompose $w$ into maximal increasing blocks of the form $1^{a_1} 2^{a_2} \cdots k^{a_k} (k+1)^{m}$. Then replace each such block with $1^{b_1} 2^{b_2} \cdots k^{b_k} (k+1)^m$, where $b_i = a_{\pi(i)}$.
\end{defn}

For instance, with $\pi=312$ and $w = 41334411112234 \in \incwords{4}$ we have $\piw = 41124412222334$.  



\begin{prop}
\label{prop:bijection}
Let $k \geq 2$ and $J \subset \{1,\ldots,k\}$.  For any $\pi \in \Sym_k$ the composite mapping 
\begin{equation}
\label{eq:trans}
 \l \xrightarrow{\bkburge} 
 w \xrightarrow{\ \ }
 \piw \xrightarrow{\bkburge^{-1}}
 \nu \xrightarrow{\ \ }
 \pil[\nu]
\end{equation}
is a bijection from $\Dstrict{J}$ to $\Dstrictpi{\pi^{-1}(J)}$ which preserves both size and number of parts in each residue class modulo $k$.  
\end{prop}


\begin{proof}
Clearly the given map is bijective from $\Dsep$ to $\Dseppi$.  Using Theorem~\ref{thm:bburge} we get, for all $i$, that
$\l$ has $a_i$ parts congruent to $i$ mod $k$ if and only if 
$\pil[\nu]$ has $b_i=a_{\pi(i)}$ parts congruent to $\pi(i)$ mod $k$. 
Clearly $\bmaj(w)=\bmaj(\piw)$ as $w$ and $\piw$ share the same $\b$-ascents.  Therefore
\begin{align*}
	|\pil[\nu]|
	&= |\nu|+ \sum_{i=1}^k (\pi(i)-i)b_i \\
	&= \Bigg(k\bmaj(\piw) - \sum_{i=1}^k(k-i)b_i\Bigg) + \sum_{i=1}^k(\pi(i)-i)b_i \\
	&= k\bmaj(w) - \sum_{i=1}^k(k-\pi(i))b_i \\	
	&= |\l|.
\end{align*}
Finally, $w$ avoids $jj$ if and only if $\piw$ avoids $\bar{j}\bar{j}$, where $\bar{j}=\pi^{-1}(j)$. So $\l \in \Dstrict{J} \iff \nu \in \Dstrict{\pi^{-1}(J)} \iff \pil[\nu] \in \Dstrictpi{\pi^{-1}(J)}$.
\end{proof}

Before concluding with an example, we observe that the bijection $\l \mapsto \pil[\nu]$ of Proposition~\ref{prop:bijection} has the following succinct description  in terms of $k$-modular diagrams:  
(1) Decompose the  diagram of $\l$ into maximal ``staircases'', (2) apply $\pi^{-1}$ to every entry, (3) sort the entries  in each staircase so  they increase weakly from bottom to top, and (4) apply $\pi$ to every entry. The resulting configuration is the diagram of $\pil[\nu]$. See Figure~\ref{fig:bijection} for a demonstration.  From this perspective it is clear that the bijection preserves size, number of parts in each residue class, and also number of parts $\leq km$ for any given $m$.

\newcommand{\YC}[1]{*(light-gray)#1}
\begin{figure}
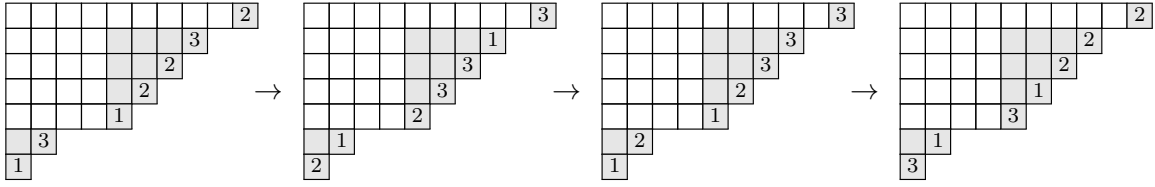

$$
\ytableausetup{smalltableaux,centertableaux}
\begin{ytableau}
\ & & & & & & & &   & \YC{2} \\
\ & & & & \YC{} & \YC{} & \YC{} & \YC{3}  \\
\ & & & & \YC{} & \YC{} & \YC{2} \\
\ & & & & \YC{} & \YC{2} \\
\ & & & & \YC{1} \\
\YC{} & \YC{3} \\
\YC{1}
\end{ytableau}
 \hspace{-1ex}\rightarrow\hspace{1ex}
\begin{ytableau}
\ & & & & & & & &   & \YC{3} \\
\ & & & & \YC{} & \YC{} & \YC{} & \YC{1}  \\
\ & & & & \YC{} & \YC{} & \YC{3} \\
\ & & & & \YC{} & \YC{3} \\
\ & & & & \YC{2} \\
\YC{} & \YC{1} \\
\YC{2}
\end{ytableau}
 \hspace{-1ex}\rightarrow\hspace{1ex}
\begin{ytableau}
\ & & & & & & & &   & \YC{3} \\
\ & & & & \YC{} & \YC{} & \YC{} & \YC{3}  \\
\ & & & & \YC{} & \YC{} & \YC{3} \\
\ & & & & \YC{} & \YC{2} \\
\ & & & & \YC{1} \\
\YC{} & \YC{2} \\
\YC{1}
\end{ytableau}
 \hspace{-1ex}\rightarrow\hspace{1ex}
\begin{ytableau}
\ & & & & & & & &   & \YC{2} \\
\ & & & & \YC{} & \YC{} & \YC{} & \YC{2}  \\
\ & & & & \YC{} & \YC{} & \YC{2} \\
\ & & & & \YC{} & \YC{1} \\
\ & & & & \YC{3} \\
\YC{} & \YC{1} \\
\YC{3}
\end{ytableau}
$$
\caption{Transforming $\l=(29,24,20,17,13,6,1)$ to $\pil[\nu]=(29,23,20,16,15,4,3)$ according to Proposition~\ref{prop:bijection}, in case $k=3$ and $\pi = 312 \in \Sym_3$. }
\label{fig:bijection}
\end{figure}

\newcommand{\numberof}[3]{{#1}\big|_{#2 ; #3}}
\newcommand{\bignumberof}[3]{{#1}\Big|_{#2 ; #3}}


\begin{exmp}
Take $k=3$ and $J=\{3\}$.  Then, for any $\pi \in \Sym_3$,  there are equally many partitions of size $n$ in the sets $\Dstrict[3]{\{3\}}$ and $\Dstrictpi[3]{\{\pi^{-1}(3)\}}$. In particular, by taking $\pi=132$ we see that the quantity $s_n=t_n$ described by Schur's Theorem is also the number of partitions $\l$ of $n$ satisfying the gap conditions~\eqref{eq:andrewsgap}.  This was originally discovered via computer search by Andrews~\cite{And-2}, who also gave a proof using generating series. A combinatorial treatment was given in~\cite{AllGor-2}, along with an extensive analysis of the $k=3$ case of Proposition~\ref{prop:bijection}. 
\end{exmp}



\section*{Acknowledgements}

This article was born out of joint work with Toma\v{z} Ko\v{s}ir and Mitja Mastnak on the ``Box Conjecture'' for commuting pairs of nilpotent matrices~\cite{IKM}, where  Burge's correspondence was found to play a central (and surprising) role.  The author would like to thank both Toma\v{z} and Mitja for helpful discussions throughout its preparation.

\bibliographystyle{plain}

\end{document}